\documentclass{article}
\usepackage{hyperref}
\usepackage[utf8]{inputenc}
\usepackage[english]{babel}
\usepackage{amsthm}
\usepackage{amsmath}
\usepackage{amsfonts}
\usepackage{xcolor}
\usepackage{enumerate}
\usepackage{tikz}
\usepackage{caption}
\usepackage{subcaption}
\usepackage[T1]{fontenc}

\newtheorem{theorem}{Theorem}[section]
\newtheorem{corollary}{Corollary}[theorem]
\newtheorem{lemma}{Lemma}[theorem]
\newtheorem{definition}{Definition}[theorem]
\newtheorem{proposition}{Proposition}[theorem]

\begin{document}
	
	\title{Granular DeGroot Dynamics -- a Model for Robust Naive Learning in Social Networks}
	
	\author{Gideon Amir\footnote{Department of Mathematics, Bar Ilan University, Israel. Email: gidi.amir@gmail.com}\and Itai Arieli\footnote{Faculty of Industrial Engineering and Management, Technion - Israel Institute of Technology. Email: iarieli@technion.ac.il} \and Galit Ashkenazi-Golan\footnote{Department of Mathematics, London School of Economics and Political Science. Email: galit.ashkenazi@gmail.com} \and Ron Peretz\footnote{Department of Economics, Bar Ilan University, Israel. Email: ron.peretz@biu.ac.il}}
	\date{\today}
	
	\maketitle
	
	\begin{abstract}
		We study a model of opinion exchange in social networks where a state of the world is realized and every agent receives a zero-mean noisy signal of the realized state. It is known from Golub and Jackson \cite{golub2010naive} that under DeGroot \cite{degroot1974reaching} dynamics agents reach a consensus that is close to the state of the world  when the network is large. The DeGroot dynamics, however, is highly non-robust and the presence of a single ``{adversarial} agent'' that does not adhere to the updating rule can sway the public consensus to any other value.
		We introduce a variant of DeGroot dynamics that we call \emph{ $\frac{1}{m}$-DeGroot}. $\frac{1}{m}$-DeGroot dynamics approximates standard DeGroot dynamics to the nearest rational number with $m$ as its denominator and like the DeGroot dynamics it is Markovian and stationary. We show that in contrast to standard DeGroot dynamics, $\frac{1}{m}$-DeGroot  dynamics is highly robust both to the presence of {adversarial} agents and to certain types of misspecifications.
	\end{abstract}
	
	% Paper body
	\section{Introduction}
	Social networks play a very important role as a conduit of information. Individuals repeatedly interact with their peers and adjust the choices they make by responding to their observed behavior. One explanation for such adaptive behavior is informational effects; i.e., individuals infer the underlying information that gives rise to the observed behavior and based on this \emph{social learning} change their own behavior (e.g., \cite{aumann1976agreeing}, \cite{geanakoplos1982we}, \cite{adlo}).  Making fully rational inferences about the private information of \emph{all} agents given the observed behavior, however, is known to be conceptually and computationally complex.\footnote{See, e.g., \cite{hkazla2019reasoning} and \cite{hkazla2021bayesian}.} In practice, agents are often unaware of the structure of the network. This unawareness prevents the agents from tracing the information path in an environment of repeated interaction.
	
	In order to overcome the complexity of rational decision making, an alternative approach has been developed. The underlying assumption of the alternative approach is that individuals follow a particular heuristics when revising their opinion. The most prominent such approach is the one proposed by DeGroot \cite{degroot1974reaching} and brought into the economics literature by DeMarzo et al. \cite{demarzo2003persuasion}. For a related work that studies belief exchange in networks see \cite{jadbabaie2012non, molavi2018theory}.
	
	The DeGroot framework comprises a social network that is represented by a (finite) graph. The vertices of the graph represent agents who are in direct relations with their neighbors in the graph. The true state of the world is represented by a real number $\mu$. In each time period $t\geq 0$ each agent $i$ obtains a subjective opinion $A_{i,t}\in\mathbb R$ regarding the state of the world. A common assumption is that the initial opinions $A_{i,0}=\mu+\varepsilon_i$, where the $\{\varepsilon_i\}_{i\in V}$ are i.i.d.\ with expectation zero. As time progresses, the agents  simultaneously revise their opinions while relying only on the opinions of their neighbors. The DeGroot updating rule asserts that an agent updates her opinion to be equal to the average of the opinions of her neighbors. 
	
	An important property of the DeGroot updating procedure, as established both by DeGroot \cite{degroot1974reaching} and by Demarzo et al. \cite{demarzo2003persuasion}, is that in the long run agents reach consensus. That is, when the entire population uses DeGroot heuristics, agents' opinions converge to the same value. Moreover, this value is a weighted average  of the initial opinions at time zero. The weights, called the \emph{measure of centrality}, are given by the stationary distribution of the random walk over the underlying network graph.
	
	When the graph is large, one can infer $\mu$ from the initial opinions $\{A_{i,0}\}_{i\in V}$, but no single agent can make such an inference by herself, since she observes her own opinion and the opinions of her neighbors only. However,  Golub and Jackson \cite{golub2010naive} demonstrated the following striking property of the DeGroot dynamics, which they termed \emph{naive learning}. Despite the simplicity of the DeGroot rule, when the graph is large and no agent is too central, the agents' opinions converge to a consensual opinion that is close to the true state of the world with high probability.
	This means that under a mild assumption on the network, 
	DeGroot heuristics enable every agent to trace back the state of the world with high probability by eliciting the so called \emph{wisdom of the crowd}.
	
	A weakness of the standard DeGroot dynamics as a tool to aggregate private information is its sensitivity to minor deviations (see Section~\ref{section fragility}). The source of such deviations may be either malicious agents, who have {the} incentive to change public opinion, or agents who suffer from behavioral biases. For example, suppose one of the agents {is adversarial, namely, it} does not follow the updating rule, while all the others do. {Specifically, say that the adversarial} agent is ``stubborn''; that is, she starts with a certain opinion and never changes it. Applying DeGroot dynamics in the presence of such a stubborn agent results in everyone adopting the opinion of the stubborn agent in any finite connected graph (see Lemma \ref{lemma bots}). Another possible modification of the model is distorted monitoring: the agents do not observe the exact opinion of each of their neighbors, but rather a slight modification of it. Say, for example, that everyone sees $x+\beta$ when the true opinion is $x$. Then, even if $\beta$ is very small (yet positive), the resulting limiting opinions of all of the agents will be $\infty$ regardless of the state of the world.
	
	The core research question we address in this paper is the following:
	does there exist 
	a simple boundedly rational heuristics that enable naive learning that is immune to the above-mentioned deviations? 
	We provide a positive answer to this question in a setting with either a large or a countably infinite graph of bounded degree.
	%where the updating procedure arrive asynchronously among agents according to a Poisson clock with a unit arrival rate.
	
	We introduce a parameterized family of heuristics, $\frac{1}{m}$-\emph{DeGroot}  that approximate DeGroot heuristics by coarsening the set of feasible opinions for the agents.
	According to the  $\frac{1}{m}$-\emph{DeGroot} updating rule, whenever an agent has a revision opportunity and the current average of the opinions of her neighbors is $y$ (which is the DeGroot updating value), 
	she updates her current opinion $x$ to be the closest rational number to $y$ that takes the form $\frac{k}{m}$ across all integers $k\in \mathbb{Z}$. That is, agents are restricted to choosing an opinion from the discrete $\frac{1}{m}$-grid,  $\{\frac{k}{m}\}_{k\in\mathbb{Z}}$, and so choose the one that is closest to\footnote{Ties are broken in favor of the agents' current belief. See discussion on the tie breaking rule in Section~\ref{section discussion}.} $y$.

	Our main results show that  $\frac{1}{m}$-DeGroot dynamics leads to approximate learning while being immune to minor deviations. Specifically, we say that an agent achieves \emph{$(\delta,\rho)$-learning} if her opinion eventually lies within $\delta$ of the true state of the world with probability at least $1-\rho$. A dynamics on an infinite graph $G$ satisfies $(\delta,\rho)$-learning if all but a finite set of agents satisfy $(\delta,\rho)$-learning. 
	
	Our first main result, Theorem~\ref{theorem: main one graph}, loosely says that if the graph $G$ is infinite with a bounded degree and a moderate growth rate,\footnote{The specific assumptions on the growth of the graph are detailed in Section~\ref{section main results}.} then for every $\delta,\rho>0$ and every $m$ large enough, $\frac{1}{m}$-DeGroot dynamics satisfies $(\delta,\rho)$-learning. Furthermore, the learning takes place even in the presence of finitely many {adversarial} agents and moderately distorted monitoring.\footnote{Establishing the immunity to monitoring distortions $\beta<\frac{1}{2m}$ is simple. To this end, we formally specify the $\frac{1}{m}$-DeGroot updating rule with an additional preliminary projection as follows: first project the opinions of the neighbors on the $\frac{1}{m}$-grid, then take the average of the projected opinions, and lastly project the average on the $\frac{1}{m}$-grid. The first projection is what eliminates the distortion. } 
	Thus, by a slight modification of standard DeGroot heuristics we offer a highly robust learning procedure. 
	%Furthermore, $(\delta,\rho)$-learning continues to hold in the presence of finitely many bots and when the agents are $0.9\varepsilon$-distorted.  
	
	Our second main result, Theorem~\ref{theorem: main family}, is a generalization of the first main result to families of graphs. Theorem~\ref{theorem: main family} can be viewed as an analog of the result of Golub and Jackson~\cite{golub2010naive} who refer to growing families of finite graphs.
	
	{Theorem~\ref{thm R} provides the idea driving our main results. The influence of each adversarial agent is limited to a certain radius around it depending on the granularity parameter $m$. Increasing $m$ has two effects. On one hand, it increases the radius of influence and thus fewer agents achieve $(\delta,\rho)$-learning. On the other hand, it increases the accuracy of learning and thus allows for smaller $\delta$ and $\rho$. In particular, as $m$ goes to infinity, the radius of influence grows to infinity and $\delta$ and $\rho$ go to zero, thus obtaining asymptotic behavior similar the the standard DeGroot dynamics.}
	
	Figure ~\ref{fig:simulation} shows two simulations: one of the $\frac 1 {100}$-DeGroot dynamics and the other of the standard DeGroot dynamics. In both simulations there is a single stubborn agent. Under the $\frac 1 {100}$-DeGroot dynamics, the influence of the stubborn agent is limited to a certain radius around it, whereas under the standard DeGroot dynamics, it is unlimited. 
	
	\begin{figure}
		\centering
		\begin{subfigure}[b]{0.8\textwidth}
			\centering
			\includegraphics[trim={0 200 0 200},clip, width=\textwidth]{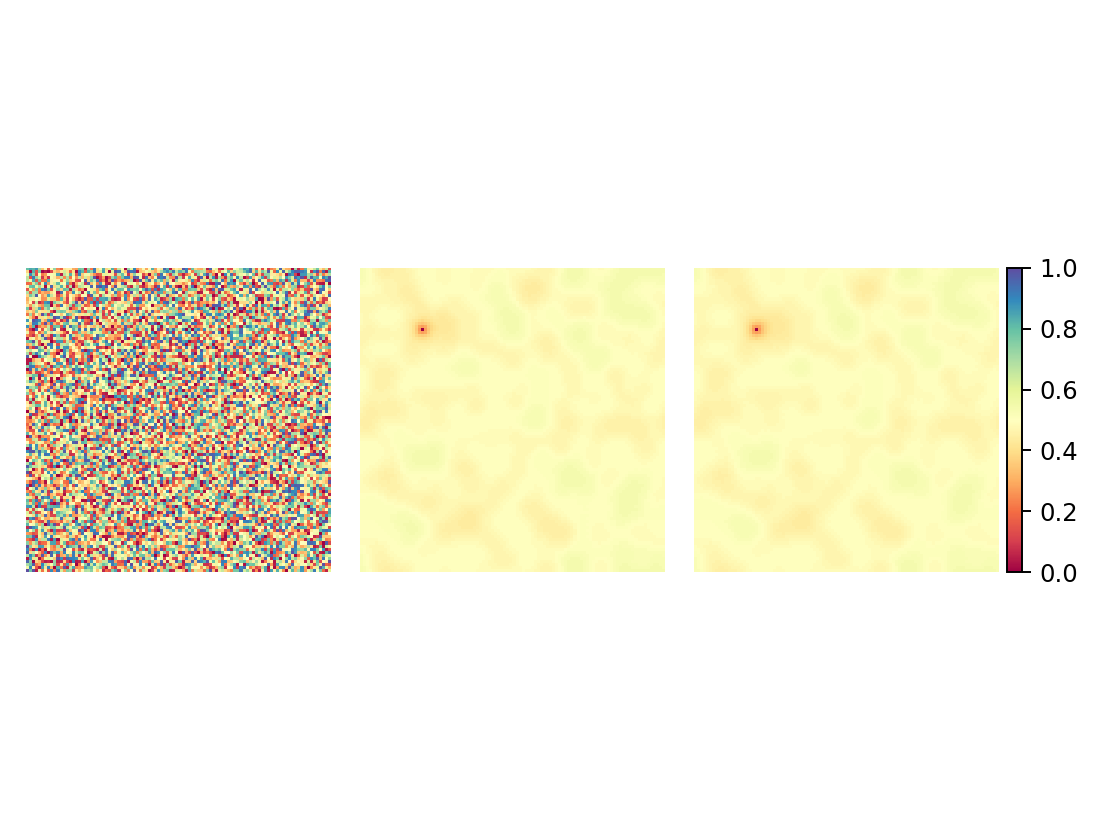}
			\caption{$\frac 1 {100}$-DeGroot}
			\label{fig:simulation granular}
		\end{subfigure}
		\begin{subfigure}[b]{0.8\textwidth}
			\centering
			\includegraphics[trim={0 200 0 200},clip,width=\textwidth]{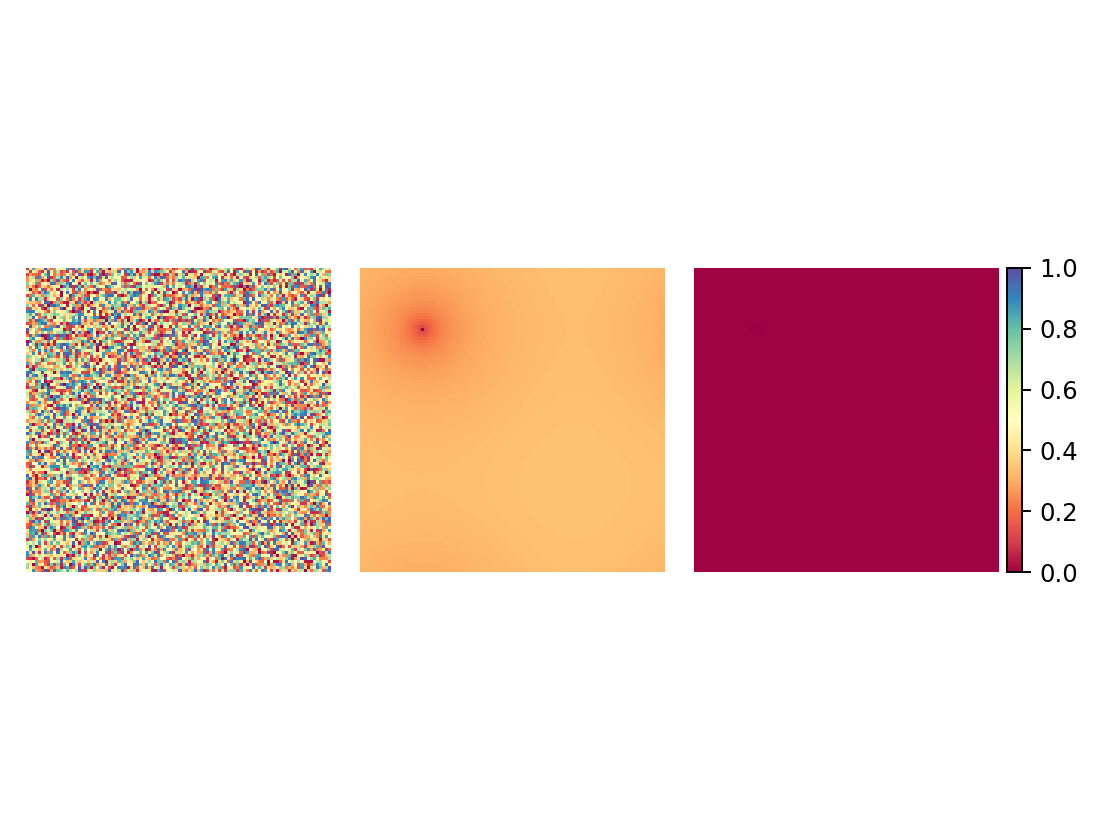}
			\caption{Standard DeGroot}
			\label{fig:simulation degroot}
		\end{subfigure}
		\caption{Simulation of two variants of the DeGroot dynamics. The underlying graph is the $100\times 100$ grid (with edges in all eight directions: N, NE, E,\ldots). Initial opinions drawn i.i.d.\ uniform $[0,1]$ except for a single stubborn agent with opinion 0. The three images (left to right) capture the opinions of all agents at times $0$, $10,000$, and $100,0000$.}
		\label{fig:simulation}
	\end{figure}
	
	The proofs of Theorems~\ref{theorem: main one graph}, \ref{theorem: main family} {and \ref{thm R}} follow from a more general result providing  sufficient conditions for a robust approximation of the DeGroot dynamics satisfying $(\delta,\rho)$-learning while being immune to {adversarial} agents and moderately distorted monitoring. In {Section~\ref{section status quo}}, we present a different dynamics, \emph{status quo biased DeGroot dynamics}, that satisfies the conditions of the general framework.
	
	The general framework is formalized in Theorems~\ref{theorem convergence} and \ref{theorem mu} in Section~\ref{section general framework}. In Section~\ref{section 1m} we show that the $\frac 1 m$-DeGroot dynamics satisfies the sufficient conditions and therefore Theorems~\ref{theorem convergence} and \ref{theorem mu} imply Theorems~\ref{theorem: main one graph} and \ref{theorem: main family}. In {Appendix}~\ref{section proof theorem conv} we prove Theorem~\ref{theorem convergence} and in {Appendix} ~\ref{section proof thm mu} we prove Theorem \ref{theorem mu}. 
	
	The rationale behind $\frac{1}{m}$-DeGroot is simple and it is compatible with the bounded rationality assumption of DeGroot dynamics. While DeGroot dynamics places no restrictions on the agents' opinions,  $\frac{1}{m}$-DeGroot dynamics restricts the agents' opinions to rational values of the form  $x\in\mathbb{Q}$ such that $mx\in\mathbb{N}$. This highly realistic restriction may be seen as a consequence
	of bounded rationality as in practice human capacity for accuracy is limited. Surprisingly, adding this natural layer of bounded rationality makes DeGroot dynamics highly robust. 
	\subsection{Related literature}
	Works that are related to ours are Yildiz et al. \cite{yildiz2013binary} and a recent note by Sadler \cite{sadler2020influence}. Their work studies a model of binary opinion dynamics based on the well-known \emph{voter model}. Under the classic voter model each agent holds a binary opinion and at every revision opportunity the agent chooses the opinion of a randomly chosen neighbor. As in this work, Yildiz et al. integrate stubborn agents that never change their opinion into the voter model. Unlike in the classic voter model, in the presence of stubborn agents the asymptotic consensus property is no longer satisfied.  Their main results characterize the asymptotic behavior and the opinion distribution of the agents as a function of the network and the opinions of the stubborn agents. The voter model has been introduced by Clifford and Sudbury \cite{clifford1973model}. The mathematical facets of the voter model have been extensively studied, e.g., by Holley and Ligget \cite{holley1975ergodic}. 
	
	The vulnerability of the DeGroot model to the presence of stubborn agents is demonstrated by Ghaderi and Srikant \cite{ghaderi2013opinion} and Acemouglu et al.\  \cite{acemoglu2010spread,acemouglu2013opinion}. In their model each agent, while revising her opinion,  may also assign a weight (which may be equal one) to her initial opinion at time zero. Their main contribution is to analyze how the presence of stubborn agents affects the limit opinion of the agents. In particular, they relate the limit opinion of the agents to the absorption probability of a random walk on the graph.\footnote{The connection between the Degroot model and random walk on a graph plays a significant role in our analysis as well.} In addition, they show how stubborn agents affect the speed at which agents' opinions converge to the limit opinion.        
	
	Finally, a series of recent papers analyzes the implications of Bayesian agents having misspecified beliefs about the environment in a social learning model. 
	%Bohren and Hauser \cite{bohren2020social} provide a general framework of misspecified beliefs and how they impact social learning outcomes. 
	Frick, Iijima, and Iishi \cite{frick2020misinterpreting} (see also \cite{frick2020stability}) provide conditions under which learning outcomes are not robust to minimal misspecifications. 
	In particular, they consider a population of Bayesian agents with a type-dependent utility who obtain information about the state of the world both from initial private signals and by observing a random sample of other agents' optimal actions over time. They show that even arbitrarily small amounts of misperception of the type distribution can generate extreme breakdowns of information aggregation, where in the long run all agents incorrectly assign probability one to some fixed state, regardless of the true realized state. 
	
	Our analysis here complements the above literature on stubborn agents by providing a robust variant of the leading non-Bayesian social learning rule in a way that is immune to the presence of stubborn {(and even adversarial)} agents. In addition, as mentioned above our learning dynamics is also immune to certain behavioral biases that can give rise to misspecifications.
	
	\section{Model and Preliminaries}
	We consider a model of opinion exchange in networks similar to that of DeGroot \cite{degroot1974reaching}. A connected graph $G=(V,E)$ (which is either finite or countably infinite) is given. Each agent $i\in V$ is represented by a vertex in the graph. Time is discrete. At time $t=0$ a state of the word $\mu\in \mathbb{R}$ is realized according to an unknown distribution $\psi\in\Delta(\mathbb{R})$.
	Thereafter, each agent {observes a noisy signal about the state of the world and sets it as its initial opinion. Specifically, the initial opinion of agent $i$ is $A_{i,0}:=\mu+x_i$, where $\{x_i\}_{i\in V}$ are zero-mean bounded (by a constant $|x_i|\leq K$) i.i.d.\ random variables}. For every agent $i$, let $N_i$ be the neighbors of agent $i$ and let $n_i=|N_i|$. Furthermore, let  $A_{i,t}\in\mathbb{R}$ be the opinion of agent $i$ at time $t\ge 0$. We assume that {$G$ has a bounded degree, namely,} there is a parameter $d>0$ such that $n_i\leq d$, for all $i\in V$.
	
	According to the DeGroot dynamics, at each time $t>0$, each agent $i$ updates her opinion by the rule
	$$A_{i,t}=\frac{1}{n_i}\sum_{j\in N_i} A_{j,t-1}:=\mathrm{Average_{j\in N_i}(A_{j,t-1})}.$$
	
	We first provide an analogue of the ``naive wisdom of crowds'' result of Golub and Jackson \cite{golub2010naive} for infinite graphs (see proof in the appendix).
	
	\begin{theorem}\label{theorem:golubjackson}
		Let $G$ be a connected infinite graph of bounded degree.
		For every agent $i\in V$ it holds that $\lim_{t\rightarrow\infty}A_{i,t}=\mu$ almost surely.
	\end{theorem}
	Golub and Jackson consider a growing sequence of finite networks with the \emph{vanishing influence} property; i.e., the largest coordinate of the stationary-measure of the Markov chain induced by the network approaches zero as the network grows. The assumption of bounded degree applied to a growing sequence of networks implies the vanishing influence property. Therefore, Theorem~\ref{theorem:golubjackson} can be viewed as an analogue of Golub and Jackson's result for infinite graphs.
	
	There is, however, a difference between finite graphs and infinite graphs. When the graph is finite the convergence of DeGroot dynamics $A_{i,t}$ as $t\to\infty$ is guaranteed\footnote{Strictly speaking, it is guaranteed that  the limits $\lim_{t\to\infty}A_{i,2t}$ and $\lim_{t\to\infty}A_{i,2t+1}$ exist, and if in addition the graph is not bipartite then the two limits are equal.} for every value of the initial opinions $\{A_{i,0}\}_{i\in V}$. For infinite graphs, there may be initial values under which $A_{i,t}$ oscillates as $t\to\infty$ {(See an example in Appendix~\ref{section example divergence})}. A simple consequence of Theorem \ref{theorem:golubjackson} is that the set of such initial values has zero probability under our assumptions on the distribution of the initial opinions.  
	
	\subsection{The Fragility of DeGroot dynamics}\label{section fragility}
	In this section we demonstrate the known fact that in finite networks the DeGroot model and the wisdom of the crowd result of Golub and Jackson \cite{golub2010naive} are highly fragile to the presence of ``stubborn'' or ``biased'' agents.  {An agent is called} \emph{stubborn} if she sets her opinion to be constant $A_{i,t}\equiv c$ over time $t$. {More generally, an agent that updates her opinion arbitrarily, subject only to the restriction that the opinion must be within a bounded set
		is call \emph{adversarial}.}
	
	\begin{lemma}\label{lemma bots}
		Consider a connected finite graph $G=(V,E)$ and a unique stubborn agent $i_0\in V$ who sets her opinion to be constant $c$. Then for every agent $i\in V$ it holds that\footnote{For completeness, we provide the proofs of this and another standard result,  Lemma~\ref{lemma distortion} below, in the appendix.
		}
		$\lim_{t\rightarrow\infty} A_{i,t}=c$.
	\end{lemma}
	
	Thus, in order to achieve the wisdom of the crowd in the standard DeGroot model, all agents have to  comply with the dynamics; even a single stubborn agent is sufficient to dramatically swing the consensus  over to her opinion.
	
	Another form of fragility in learning dynamics is that of agents who suffer from a bias in aggregating others' opinions. For each agent $i\in V$ of degree $d_i=|N_i|$ one may view an updating heuristic at time period $t>0$ as a function $g_{i,t}:\mathbb{R}^{(d_i+1)(t-1)}\rightarrow\mathbb{R}$. Denote $(i):=N_i\cup\{i\}$ {and  $A_{(i),<t}:=A_{|(i)\times \{0,\ldots,t-1\}}$, namely,  the restriction of $A$ to $(i)\times \{0,\ldots,t-1\}$.} According to the heuristic $g=\{g_{i,t}\}$, agent $i$ is supposed to update her opinion such that
	\[
	A_{i,t}=g_{i,t}({A_{(i),<t}}).
	\]
	We say that $i$  suffers from \emph{$\beta$-distorted monitoring} with respect to $g$ if in every time period $t>0$ 
	$$A_{i,t}=g_{i,t}({A'_{(i),<t}}),$$
	for some ${A'_{(i),<t}}$, such that for every $0\leq s<t$,
	\[
	|A'_{j,s}-A_{j,s}|\leq
	\begin{cases}
		\beta& j\in N_i,\\
		0& j=i.
	\end{cases}
	\]
	Thus, agent $i$ updates her opinion in accordance with $g$ but  may perceive some of the opinions of her neighbors differently than they actually are. For example, she may add the value $\beta$ to the opinions of her neighbors, or, vice versa, subtract $\beta$ from the opinions of all of her neighbors. 
	
	The following lemma shows that no matter how small $\beta>0$ is, a single agent who suffers from $\beta$-distorted monitoring with respect to the DeGroot updating rule is sufficient to sway  public opinion away from the true state of the world to any other point.
	
	\begin{lemma}\label{lemma distortion}
		Consider a connected finite graph $G=(V,E)$. Let  $i_0\in V$, $A_{i,0}\in \mathbb R^V$, $\beta>0$,  $\underline l\leq \overline l$. Suppose one agent suffers from $\beta$-distorted monitoring with respect to DeGroot dynamics and all other agents follow DeGroot dynamics precisely. Then there exists a suitable distortion that leads to opinions $\{A_{i,t}\}_{i\in V,t\in\mathbb Z_+}$ that satisfy $\liminf_{t\to\infty}A_{i_0,t}=\underline l$ and $\limsup_{t\to\infty}A_{i_0,t}=\overline l$. 
	\end{lemma}
	
	\section{Main Results}\label{section main results}
	Our main goal is to introduce a robust version of DeGroot dynamics. We define a parameter-dependent family of heuristics.
	
	\begin{definition}\label{definition 1m DeGroot} The \emph{ $\frac 1m$-DeGroot} updating rule is defined as follows. In each time period $t>0$,
		each agent $i\in V$ observes  $A_{(i),t-1}$ and updates her opinion in three steps:
		\begin{enumerate}[(i)] 
			\item the observed opinions are rounded to the nearest value in the $\frac 1 m$-grid;
			\item the average of the rounded observed opinions is computed; 
			\item the result is rounded to the nearest value in the $\frac 1 m$-grid{, and that becomes the new opinion}.
		\end{enumerate}
		In steps (i) and (iii) ties between two alternatives are broken in favor of the one closer to $A_{i,t-1}$.
	\end{definition}
	
	Note that the definition of $\frac 1m$-DeGroot dynamics guarantees that the induced {opinions} take values in the $\frac 1m$-grid. Consequently, applying step (i) of the definition is needed only when considering distorted monitoring. In addition, two properties hold.  The first is   
	\begin{align*}
		A_{i,t}\in \arg\min\{|\tfrac{k}{m}-\mathrm{Average}_{j\in N_i}(A_{j,t-1})|: k\in\mathbb{Z}\},    
	\end{align*}
	where ties between two alternatives are broken in favor of the one that is closer to $A_{i,t-1}$. 
	The second is that any distortion smaller than $\frac 1 {2m}$ is eliminated by step (i). 
	
	To obtain our main results we make an assumption on the growth rate of the underlying network. To formulate this assumption we need the following definitions. We say that a graph $G=(V,E)$ is \emph{majorized by} a function $f\colon\mathbb R_+\to \mathbb R_+$ if for every $i\in V$ and every $r\in\mathbb N$, the ball of radius $r$ around $i$, denoted by $B(i,r)\subset V$, satisfies $|B(i,r)|\leq f(r)$. We say that $f$ has a \emph{stretched exponential} growth rate if there exist $\alpha>0$ and $r_0>0$ such that $f(r)>\exp(r^\alpha)$ for all $r\geq r_0$. We say that $f$ has a \emph{stretched sub-exponential} growth rate if for all $\alpha>0$, there exists $r_0>0$, such that $f(r)<\exp(r^\alpha)$, for all $r\geq r_0$. 
	We also say that $f$ is a stretched (sub-) exponential function to mean the same thing. {Our main results, Theorems~\ref{theorem: main one graph}, \ref{theorem: main family}, and \ref{thm R}, will assume that the underlying network is majorized by a stretched sub-exponential function. Other intermediate results hold under weaker assumptions.}
	
	{Due to the granularity of the opinions, we cannot have exact learning of the state of nature. Hence, we define the following notion of approximate learning.}
	
	\begin{definition}
		Consider a dynamics $A_{i,t}$ on a graph $G=(V,E)$ { and $\delta,\rho>0$}. We say that $i\in V$ satisfies $(\delta,\rho)$-\emph{learning} if 
		\[
		\Pr(\mu-\delta\leq \liminf_{t\rightarrow\infty}A_{i,t} \leq \limsup_{t\rightarrow\infty}A_{i,t}\leq \mu +\delta)\geq 1-\rho.
		\]
		We say that the dynamics satisfies $(\delta,\rho)$-\emph{learning} on $G$ if all but a finite set of agents satisfy $(\delta,\rho)$-learning. 
	\end{definition}
	We can now state our first main result.
	\begin{theorem}\label{theorem: main one graph}
		Let $G$ be an infinite connected graph of bounded degree majorized by a {{stretched}} sub-exponential function. For every $\delta,\rho>0$, there exists $m>0$ {(depending only on $\delta$, $\rho$, and the majorizing function)} such that $\frac 1 m$-DeGroot
		dynamics satisfies $(\delta,\rho)$-learning on $G$. 
		
		Furthermore, $(\delta,\rho)$-learning on $G$ continues to hold even in the presence of finitely many {adversarial} agents. 
	\end{theorem}
	
	We next take the approach of Golub and Jackson and define learning on a family of graphs. This approach enables us to extend our results to finite graphs since it applies to both finite and infinite graphs.
	We say that a family of graphs  $\mathcal{G}$ is majorized by a function $f\colon\mathbb R_+\to \mathbb R_+$ if every graph $G\in\mathcal{G}$ is majorized by $f$. In addition, we say that the family $\mathcal{G}$ has a bounded degree if there exists $d>0$ such that the degree of any agent in any graph $G\in\mathcal G$ is at most $d$.
	
	{Approximate learning on a family of graphs is defined as follows.}
	\begin{definition}
		Consider a family of graphs $\mathcal{G}$. Suppose that we have a family of dynamics, one on each graph $G\in\mathcal{G}.$  We say that the family of dynamics satisfies $(\delta,\rho)$-\emph{learning} on $\mathcal G$ if there exists a number $n$ such that for every $G\in\mathcal G$ all but $n$ agents satisfy $(\delta,\rho)$-learning.
	\end{definition}
	
	We are now ready to state our second main theorem, which generalizes our first main result to families of graphs.
	
	\begin{theorem}\label{theorem: main family}
		Let $\mathcal{G}$ be a family of connected graphs of bounded degree majorized by a stretched sub-exponential function. For every $\delta,\rho>0$, there exists $m>0$ {(depending only on $\delta$, $\rho$, and the majorizing function)} such that the $\frac 1 m$-DeGroot
		dynamics satisfies $(\delta,\rho)$-learning on $\mathcal{G}$. 
		
		Furthermore, $(\delta,\rho)$-learning on $\mathcal{G}$ continues to hold even in the presence of a bounded number of {adversarial} agents.  
	\end{theorem}
	
	\section{Outline of the Technical Contribution}\label{section technical contribution}
	
	Our technical contribution builds upon and bridges two existing bodies of research: the theories of DeGroot and majority dynamics. According to majority dynamics agents hold binary opinions and update their opinions to match the majority among their neighbors while breaking ties in favor of one's own opinion. Majority dynamics is a special case of $\frac 1 m$-DeGroot dynamics with $m=1$. DeGroot dynamics is the other extreme, where $m\to\infty$. DeGroot dynamics is closely related to the voters' model, according to which an agent updates her opinion by adopting the opinion of a random neighbor. The relation between DeGroot dynamics and the voters' model is that the former is the expectation of the latter (see Section~\ref{section degroot}). Both majority dynamics and DeGroot dynamics (and the related voters' model dynamics) have been extensively studied in the past decades. In this section we outline how we build upon the different theories developed for each one of the two models while generalizing these theories, connecting them, and developing new techniques.

	DeGroot dynamics and the majority dynamics each has a different desirable feature. DeGroot dynamics converges to the average of the initial opinions of the entire population, thus leads to learning the true state of the world. The majority dynamics stabilizes and thus limits the influence of stubborn agents and distorted monitoring to time up to stabilization. The $\frac 1 m$-DeGroot dynamics balances these two features: as $m$ grows the accuracy of learning increases, while the time until stabilization increases as well and hence so does the influence of stubborn agents and distorted monitoring. The key technical contribution is quantifying these effects and tuning $m$ accordingly.

	Since Theorem~\ref{theorem: main one graph} is a special case of Theorem~\ref{theorem: main family}, it is sufficient to prove Theorem~\ref{theorem: main family}. Theorem~\ref{theorem: main family}, in turn, readily follows from the following theorem,  which roughly says that the influence of any stubborn agent is limited to a certain radius around her and outside that radius approximate learning takes place.
	
	\begin{theorem}\label{thm R}
		Let $f\colon\mathbb R_+\to\mathbb R_+$ be a stretched sub-exponential function, $d>0$, and $\delta,\rho>0$. There exist $m\in\mathbb Z_+$ and $R>0$
		such that $\frac 1 m$-DeGroot dynamics on any connected graph that is majorized by $f$, has a degree at most $d$, and a radius at least $R$ has the property that
		every agent $i$ such that $B(i,R)$ does not include an {adversarial} agent satisfies $(\delta,\rho)$-learning.
	\end{theorem}
	
	{The proof of Theorem~\ref{thm R} follows several steps.
		\begin{enumerate}[Step 1.]
			\item Reduction to the case where the adversarial agents are stubborn.
			\item Convergence of opinion in the presence of stubborn agents. This is proved using a Lyapunov function.
			\item Without stubborn agents, the limit opinion of each agent $i$ is close to $\mu$ with high probability. The analysis is done by carefully choosing a time $T$ such that: 
			\begin{enumerate}
				\item $T$ is small enough for the opinion of agent $i$ to stay close to $i$'s opinion  in standard DeGroot dynamics up to time $T$;
				\item $T$ is large enough so that $i$'s opinion in the standard DeGroot dynamics at time $T$ is close to $\mu$ with high probability;
				\item Using the Lyapunov function to show that the opinion of agent $i$ does not change much from time $T$ and on. This property holds even in the presence of stubborn agents.
			\end{enumerate}
			\item Away from stubborn agents, the limit opinion is close to $\mu$ with high probability. This is concluded by noting that steps (3a)--(3c) hold for any agent whose distance from any stubborn agent is greater than $T$.
		\end{enumerate}
		
		We next elaborate on the above steps. The full proof in further generality is given in Section~\ref{section general framework} and in the Appendix.
		\subsection{Reduction from adversarial to stubborn agents}
		Theorem~\ref{thm R} (and therefore all our other theorems) reduces to the case where the adversarial agents are stubborn. 
		To see this, consider two modifications of the subject dynamics $A_{i,t}$. One modification $\underline A_{i,t}$ is derived  from $A_{i,t}$ by replacing the adversarial agents with stubborn agents whose opinion is $\inf_{i\in V,t\in \mathbb N}A_{i,t}$, and the other modification $\overline{A}_{i,t}$ is similarly obtained with $\sup_{i\in V,t\in\mathbb N}A_{i,t}$. By induction on $t\geq 0$, we $\underline A_{i,t}\leq A_{i,t}\leq\overline {A}_{i,t}$, for every $i\in V$ and $t\in\mathbb N$. Therefore, Theorem~\ref{thm R} applied to $\underline{A}_{i,t}$ and $\overline A_{i,t}$ implies Theorem~\ref{thm R} for $A_{i,t}$.
		
		The proof of Theorem~\ref{thm R} will assume that the adversarial agents are stubborn without loss of generality. The first step of the proof is to establish convergence.}
	\subsection{Convergence of opinions}
	Majority dynamics and DeGroot dynamics both have the following convergence property, which we call \emph{alternate convergence}: the limits $\lim_{t\to \infty}A_{i,2t+\iota}$, $\iota\in \left\{0,1\right\}$ exist for any agent $i$ and any vector of initial opinions. DeGroot dynamics satisfies this property for the class of finite networks,\footnote{See \cite{degroot1974reaching} for the case of finite graphs and compare to Theorem~\ref{theorem:golubjackson} for infinite graphs, which has the additional requirement that the initial opinions are drawn i.i.d.} whereas majority dynamics satisfies it for the larger class of (possibly infinite) networks of sub-exponential growth rate\footnote{An infinite graph $G=(E,V)$ has a sub-exponential growth rate if for some (equivalently every) vertex $i\in V$ the function $f(r)=|B(i,r)|$ is sub-exponential, namely, $\lim_{r\to\infty}\frac 1 r\log f(r)=0$. This is a weaker assumption than the growth rate assumption of our main results in two respects: (i) the function is just sub-exponential rather than stretched sub-exponential;(ii) every vertex may have a different  growth function and the supremum of all these function need not be sub-exponential.} (see \cite{ginosar2000majority}). In this respect $\tfrac 1 m$-DeGroot resembles majority dynamics as we informally explain next. 
	
	A standard technique for establishing convergence for both majority and DeGroot dynamics is to consider a quantity called \emph{energy}. The key properties of energy are that it is (i) non-negative, and (ii) non-increasing as the dynamics evolves. Let us first consider a finite network $G=(V,E)$.  Classically, energy is defined as the sum of the squared differences of opinions between neighboring agents, \[\sum_{(i,j)\in E}(A_{i,t}-A_{j,t})^2.\]

	The idea of using energy is as follows. Energy as a function of the opinion of a single agent $i\in V$ (while keeping all other opinions fixed) is a quadratic polynomial that assumes its minimum at the average of the opinions of the neighbors of $i$, which we call the ``DeGroot value.'' Explicitly, up to a constant term, this polynomial is
	$P_{i,t}(X):=\sum_{j\in N_i}(X-A_{j,t})^2$. Therefore, when a single agent updates her opinion by the DeGroot rule the energy decreases. In fact, the energy decreases even when the opinion merely moves \emph{closer} to the DeGroot value, as is the case with majority dynamics and more generally with $\tfrac 1 m$-DeGroot dynamics.

	However, the energy, as classically defined above, is not guaranteed to decrease if all agents update their opinions simultaneously. To overcome this problem Ginosar and Holzman  \cite{ginosar2000majority} introduced the following variant of the energy function that relies on two consecutive time periods, \[\mathcal{E}(t):=\sum_{\substack{(i,j)\in V\times V :\\ \left\{i,j\right\}\in E}}(A_{i,t+1}-A_{j,t})^2.\]     
	The idea behind Ginosar--Holzman's energy is that it can be decomposed into sums of the above polynomials in two different manners (see Figure~\ref{fig:pt}), \begin{equation}\label{eq gh decomposition}
		\mathcal{E}(t)=\sum_{i\in V}P_{i,t}(A_{i,t+1})=\sum_{i\in V}P_{i,t+1}(A_{i,t}).   
	\end{equation}
	It follows that in order to show that $\mathcal{E}(t)\leq \mathcal{E}(t-1)$, it is sufficient to show that $P_{i,t}(A_{i,t+1})\leq P_{i,t}(A_{i,t-1})$ as illustrated in Figure~\ref{fig:et decreases}. The last inequality holds since $\tfrac 1 m$-DeGroot minimizes the distance to the DeGroot value along the $\tfrac 1 m$-grid.

	\begin{figure}
		\centering
		\begin{subfigure}[b]{0.5\textwidth}
			\centering     \resizebox{1\textwidth}{!}{%
				\begin{tikzpicture}
					\filldraw[black] (0,3) circle (2pt) node[anchor=south] {};
					\filldraw[black] (2,3) circle (2pt) node[anchor=south] {};
					\filldraw[black] (4,3) circle (2pt) node[above=3mm] {$A_{i,t+1}$};
					\filldraw[black] (6,3) circle (2pt) node[anchor=south] {};
					\filldraw[black] (8,3) circle (2pt) node[anchor=south] {};
					
					\filldraw[black] (0,0) circle (2pt) node[anchor=north] {};
					\filldraw[black] (2,0) circle (2pt) node[below=2mm] {$A_{i-1,t}$};
					\filldraw[black] (4,0) circle (2pt) node[anchor=north] {};
					\filldraw[black] (6,0) circle (2pt) node[below=2mm] {$A_{i+1,t}$};
					\filldraw[black] (8,0) circle (2pt) node[anchor=north] {};
					\filldraw[black] (8.25,0) circle (0.4pt);
					\filldraw[black] (8.5,0) circle (0.4pt);
					\filldraw[black] (8.75,0) circle (0.4pt);
					\filldraw[black] (8.25,3) circle (0.4pt);
					\filldraw[black] (8.5,3) circle (0.4pt);
					\filldraw[black] (8.75,3) circle (0.4pt);
					\filldraw[black] (-0.25,3) circle (0.4pt);
					\filldraw[black] (-0.5,3) circle (0.4pt);
					\filldraw[black] (-0.75,3) circle (0.4pt);
					\filldraw[black] (-0.25,0) circle (0.4pt);
					\filldraw[black] (-0.5,0) circle (0.4pt);
					\filldraw[black] (-0.75,0) circle (0.4pt);
					
					\draw (0,3)--(2,0);
					\draw (0,0)--(2,3);
					\draw (2,3)--(4,0);
					\draw (4,3)--(2,0);
					\draw (4,3)--(6,0);
					\draw (6,3)--(4,0);
					\draw (6,3)--(8,0);
					\draw (8,3)--(6,0);
					
					\draw [dotted] plot [smooth ] coordinates { (0,3.4) (2.2,0) (1.7,0) (0,2.3)};
					
					\draw [dotted] plot [smooth cycle] coordinates {(-0.3,0) (2,3.4) (4.2,0) (3.7,0) (2,2.3) (0.3,0)};
					
					\draw [dashed] plot [smooth cycle] coordinates {(1.7,0) (4,3.4) (6.2,0) (5.7,0) (4,2.3) (2.3,0)};
					
					\draw [dotted] plot [smooth cycle] coordinates{(3.7,0) (6,3.4) (8.2,0) (7.7,0) (6,2.3) (4.3,0)};
					
					\draw [dotted] plot [smooth ] coordinates{  (8,2.3) (6.3,0)(5.7,0)(8,3.4) };
				\end{tikzpicture}
			}%
			
			\caption{The dashed area corresponds to $P_{i,t}(A_{i,t+1})$.}
		\end{subfigure}
		\begin{subfigure}[b]{0.5\textwidth}
			\centering     
			\resizebox{1\textwidth}{!}{%
				\begin{tikzpicture}
					\filldraw[black] (0,3) circle (2pt) node[anchor=south] {};
					\filldraw[black] (2,3) circle (2pt) node[above=3mm] {$A_{i-1,t+1}$};
					\filldraw[black] (4,3) circle (2pt) node[above=3mm] {};
					\filldraw[black] (6,3) circle (2pt) node[above=3mm] {$A_{i+1,t+1}$};
					\filldraw[black] (8,3) circle (2pt) node[anchor=south] {};
					
					\filldraw[black] (0,0) circle (2pt) node[anchor=north] {};
					\filldraw[black] (2,0) circle (2pt) node[below=2mm] {};
					\filldraw[black] (4,0) circle (2pt) node[below=3mm] {$A_{i,t}$};
					\filldraw[black] (6,0) circle (2pt) node[below=2mm] {};
					\filldraw[black] (8,0) circle (2pt) node[anchor=north] {};
					\filldraw[black] (8.25,0) circle (0.4pt);
					\filldraw[black] (8.5,0) circle (0.4pt);
					\filldraw[black] (8.75,0) circle (0.4pt);
					\filldraw[black] (8.25,3) circle (0.4pt);
					\filldraw[black] (8.5,3) circle (0.4pt);
					\filldraw[black] (8.75,3) circle (0.4pt);
					\filldraw[black] (-0.25,3) circle (0.4pt);
					\filldraw[black] (-0.5,3) circle (0.4pt);
					\filldraw[black] (-0.75,3) circle (0.4pt);
					\filldraw[black] (-0.25,0) circle (0.4pt);
					\filldraw[black] (-0.5,0) circle (0.4pt);
					\filldraw[black] (-0.75,0) circle (0.4pt);
					
					\draw (0,3)--(2,0);
					\draw (0,0)--(2,3);
					\draw (2,3)--(4,0);
					\draw (4,3)--(2,0);
					\draw (4,3)--(6,0);
					\draw (6,3)--(4,0);
					\draw (6,3)--(8,0);
					\draw (8,3)--(6,0);
					
					\draw [dotted] plot [smooth ] coordinates { (0,-0.4) (2.2,3) (1.7,3) (0,0.7)};
					
					\draw [dotted] plot [smooth cycle] coordinates {(-0.3,3) (2,-0.4) (4.2,3) (3.7,3) (2,0.7) (0.3,3)};
					
					\draw [dashed] plot [smooth cycle] coordinates {(1.7,3) (4,-0.4) (6.2,3) (5.7,3) (4,0.7) (2.3,3)};
					
					\draw [dotted] plot [smooth cycle] coordinates{(3.7,3) (6,-0.4) (8.2,3) (7.7,3) (6,0.7) (4.3,3)};
					
					\draw [dotted] plot [smooth ] coordinates{  (8,0.7) (6.3,3)(5.7,3)(8,-0.4) };
				\end{tikzpicture}
			}%
			\caption{The dashed area corresponds to $P_{i,t+1}(A_{i,t})$.}
		\end{subfigure}
		\caption{\label{fig:pt}An illustration of Ginosar--Holzman's energy on the line graph $\mathbb Z$.}
	\end{figure}

	\begin{figure}
		\centering
		\resizebox{0.5\textwidth}{!}{%
			\begin{tikzpicture}
				\filldraw[black] (0,-3) circle (0.5pt) node[anchor=south] {};
				\filldraw[black] (2,-3) circle (0.5pt) node[above=3mm] {};
				\filldraw[black] (4,-3) circle (2pt) node[below=3mm] {$A_{i,t-1}$};
				\filldraw[black] (6,-3) circle (0.5pt) node[above=3mm] {};
				\filldraw[black] (8,-3) circle (0.5pt) node[anchor=south] {};
				\filldraw[black] (0,3) circle (0.5pt) node[anchor=south] {};
				\filldraw[black] (2,3) circle (0.5pt) node[anchor=south] {};
				\filldraw[black] (4,3) circle (2pt) node[above=3mm] {$A_{i,t+1}$};
				\filldraw[black] (6,3) circle (0.5pt) node[anchor=south] {};
				\filldraw[black] (8,3) circle (0.5pt) node[anchor=south] {};
				
				\filldraw[black] (0,0) circle (0.5pt) node[anchor=north] {};
				\filldraw[black] (2,0) circle (2pt) node[below=2mm] {};
				\filldraw[black] (4,0) circle (0.5pt) node[anchor=north] {};
				\filldraw[black] (6,0) circle (2pt) node[below=2mm] {};
				\filldraw[black] (8,0) circle (0.5pt) node[anchor=north] {};

				%\draw (0,3)--(2,0);
				%\draw (0,0)--(2,3);
				%\draw (2,3)--(4,0);
				\draw (4,3)--(2,0);
				\draw (4,3)--(6,0);
				%\draw (6,3)--(4,0);
				%\draw (6,3)--(8,0);
				%\draw (8,3)--(6,0);
				
				\draw [dotted] plot [smooth ] coordinates { (0,3.4) (2.2,0) (1.7,0) (0,2.3)};
				
				\draw [dotted] plot [smooth cycle] coordinates {(-0.3,0) (2,3.4) (4.2,0) (3.7,0) (2,2.3) (0.3,0)};
				
				\draw [loosely dashed] plot [smooth cycle] coordinates {(1.7,0) (4,3.4) (6.2,0) (5.7,0) (4,2.3) (2.3,0)};%
				
				\draw [dotted] plot [smooth cycle] coordinates{(3.7,0) (6,3.4) (8.2,0) (7.7,0) (6,2.3) (4.3,0)};
				
				\draw [dotted] plot [smooth ] coordinates{  (8,2.3) (6.3,0)(5.7,0)(8,3.4) };
				%\draw (0,0)--(2,-3);
				%\draw (0,-3)--(2,0);
				\draw (2,0)--(4,-3);
				%\draw (4,0)--(2,-3);
				%\draw (4,0)--(6,-3);
				\draw (6,0)--(4,-3);
				%\draw (6,0)--(8,-3);
				%\draw (8,0)--(6,-3);
				\draw [dotted] plot [smooth ] coordinates { (0,-3.4) (2.1,0) (1.8,0) (0,-2.3)};
				
				\draw [dotted] plot [smooth cycle] coordinates {(-0.3,0) (2,-3.4) (4.2,0) (3.7,0) (2,-2.3) (0.3,0)};
				\draw [loosely dashed] plot [smooth cycle] coordinates {(1.7,0) (4,-3.4) (6.2,0) (5.7,0) (4,-2.3) (2.3,0)};%
				\draw [dotted] plot [smooth cycle] coordinates{(3.7,0) (6,-3.4) (8.2,0) (7.7,0) (6,-2.3) (4.3,0)};
				\draw [dotted] plot [smooth ] coordinates{  (8,-2.3) (6.3,0)(5.7,0)(8,-3.4) };
				\draw [->,densely dashed, red, thick, ] (4,-2.8) to [out=60,in=-60] (4,2.8);
			\end{tikzpicture}
		}%
		\caption{\label{fig:et decreases}A sufficient condition for Ginosar--Holzman's energy not to increase is that  $A_{i,t+1}$ is always closer to $g_D$ compared to $A_{i,t-1}$.}
	\end{figure}

	More generally, let $g_D$ be the DeGroot value (of agent $i$ at time $t$). Any dynamics that satisfies $|A_{i,t+1}-g_D|\leq |A_{i,t-1}-g_D|$ or, equivalently, \[0\leq (A_{i,t-1}-g_D)^2-(A_{i,t+1}-g_D)^2,\]
	for every agent $i$ at any time $t$, has  non-increasing energy. In order to establish convergence, we need a slightly stronger condition that will prevent the possibility that $A_{i,t-1}$ and $A_{i,t+1}$ bounce back and forth around $g_D$. The following condition serves this purpose 
	\begin{equation}\label{eq eta}
		\eta\cdot|A_{i,t-1}-A_{i,t+1}|\leq (A_{i,t-1}-g_D)^2-(A_{i,t+1}-g_D)^2,
	\end{equation} 
	for some $\eta>0$. Note that $\tfrac 1 m$-DeGroot dynamics does not satisfy \eqref{eq eta} since it may happen that $A_{i,t-1}$ and $A_{i, t+1}$ are at exactly the same distance from $g_D$ while $A_{i,t-1}\neq A_{i, t+1}$ (since the tiebreaker is $A_{i,t}$ rather than $A_{i,t-1}$). Ginosar and Holzman  tackled this issue {(}in the context of majority dynamics{)} by introducing self-loops{\footnote{A self-loop is an edge from a vertex to itself.}} at vertices of even degree{. This modification prevents ties while having the same effect as resorting to $A_{i,t}$ as a tiebreaker. Thus, the dynamics remains unchanged. The only change is in the energy function which becomes strictly decreasing whenever an agent changes its opinion.} In the more general context of $\frac 1 m$-DeGroot dynamics, we apply an analogous solution that involves generalizing the situation to weighted graphs and adding auxiliary self-loops of small weight.   
	
	We turn to consider infinite networks of sub-exponential growth rate. The problem with infinite networks is that the energy may be infinite. Tamuz and Tesler \cite{tamuz2015majority} resolve{d} this problem by using a weighted version of energy, introducing the following Lyapunov function,  \[L(t):=\sum_{\substack{(i,j)\in V\times V :\\ \left\{i,j\right\}\in E}}w(i,j)(A_{i,t+1}-A_{j,t})^2,\]
	where $\sum_{i,j}w(i,j)<\infty$ and the ratio $\tfrac{w(i,j)}{w(j,k)}$ is close to one. 
	
	The analogous decomposition of the Lyapunov function involves analogous polynomials of the form $Q_{i,t}(X):=\sum_{j\in N_i}w(i,j)(X-A_{j,t})^2$ whose minimum is no longer attained at the DeGrrot value itself but rather at a nearby point. This change requires changing condition \eqref{eq eta} into the following condition, 
	\begin{equation}\label{eq gamma eta} 
		\eta\cdot|A_{i,t-1}-A_{i,t+1}|\leq (A_{i,t-1}-v)^2-(A_{i,t+1}-v)^2, 
	\end{equation} 
	for every $v\in [g_D-\gamma,g_D+\gamma]$, for some $\gamma, \eta>0$. 
	
	$\tfrac 1m$-DeGroot dynamics on a graph of bounded degree satisfies condition~\eqref{eq gamma eta} with $\gamma,\eta=\tfrac{1}{\mathcal{O}(m)}$ (see Lemma~\ref{lemma 1m robustness}).
	More generally, {we call} a dynamics that satisfies condition \eqref{eq gamma eta} \emph{$(\gamma, \eta)$-robust}. {We call a} dynamics \emph{robust} if it is $(\gamma,\eta)$-robust for some $\gamma, \eta>0$.
	
	Robustness is a sufficient condition for alternate convergence of bounded dynamics on graphs of sub-exponential growth rate (see Proposition~\ref{proposition convergence}). The reason for that is that on such graphs it is possible to define Lyapunov functions $L_i(t)$, one for each agent $i$, such that the $(\gamma,\eta)$-robustness guarantees the following properties:
	\begin{enumerate}[(i)]
		\item $L_i(t)$ is non-negative and non-increasing in $t$ (namely, $0\leq L_i(t+1)\leq L_i(t)$, for every $t\geq 0$);
		\item every time the opinion of agent $i$ changes by $x$ between two other time periods, the function $L_i(t)$ decreases by at least $\eta x$ (namely, $L_i(t-1)-L_i(t)\geq\eta|A_{i,t-1}-A_{i,t+1}|$). 
	\end{enumerate}
	These properties imply that for any $\iota\in\{0,1\}$, the variation of $\{A_{i,2t+\iota}\}_{t=0}^\infty$ is bounded from above (by $L_i(0)$) and, therefore, $A_{i,2t+\iota}$ converges as $t\to\infty$.
	
	The second part of the proof of Theorem~\ref{thm R} shows that approximate learning occurs.
	
	\subsection{Approximate learning}
	
	In this part we show that the limiting opinion of any agent who is sufficiently far from any stubborn agent is close to the true state of the world, $\mu$, with probability close to one. 
	
	In addition to robustness, the following two properties of $\tfrac 1m$-DeGroot dynamics will play a significant role in the proof. 
	
	\begin{itemize}
		\item \emph{monotonicity}: the opinion of agent $i$ at time $t$ is a monotonic function of the history up to time $t$; and
		\item \emph{$\tfrac 1m$-averaging}:
		$|A_{i,t+1}-g_D|\leq\tfrac 1m$.
	\end{itemize} 
	
	We sketch the outline of the proof. Fix an agent $i$ and consider the dynamics of $A_{i,t}$ induced by $\tfrac 1m$-DeGroot updating rule. The analysis is divided into two stages: up to some time $T$ (which has to be carefully determined) and from time $T$ onward. 
	
	For up to time $T$, we compare $A_{i,t}$ to standard DeGroot dynamics $D_{i,t}$. The two dynamics are coupled via sharing the same initial opinions. On the one hand, monotonicity and $\tfrac 1m$-averaging guarantee that $|A_{i,T+\iota}-D_{i,T+\iota}|\leq \tfrac{T+\iota}{m}$ (which is small for $T$ that is not too large compared to $m$). On the other hand, analysis of DeGroot dynamics shows that $D_{i,T+\iota}$ is close to $\mu$ with probability close to one (as long as $T$ is large enough). By carefully choosing $T$, we guarantee that both conditions hold, and thus $A_{i,T+\iota}$ is close to $\mu$ with probability close to one. 
	
	It remains to verify that from time $t=T$ onward $A_{i,t}$ does not change much. To this end, we slightly modify $A_{T+\iota}=(A_{j,T+{\iota}})_{j\in V}$ upwards and downwards to obtain $A'_{T+\iota}\leq A_{T+\iota}\leq A''_{T+\iota}$ such that $|A'_{i,T+\iota}-A''_{i,T+\iota}|$ is small with probability close to one. In addition, $A'_{T+\iota}$ and $A''_{T+\iota}$ are defined such that the respective Lyapunov functions $L_i'(T)$ and $L_i''(T)$ are small with probability close to one. For times $t>T$, we define $A'_t$ and $A''_t$ by assuming the $\tfrac 1m$-DeGroot updating rule. The monotonicity property implies that $A'_{i,t}\leq A_{i,t}\leq A''_{i,t}$, for any $t\geq T$; therefore the  limits
	\begin{align*}
		Z_i &= \lim_{t\to\infty}(A_{i,2t+\iota})_{\iota=0,1},\\
		Z'_i &= \lim_{t\to\infty}(A'_{i,2t+\iota})_{\iota=0,1},\text{ and}\\
		Z''_i &= \lim_{t\to\infty}(A''_{i,2t+\iota})_{\iota=0,1}\\
	\end{align*}
	satisfy $Z_i'\leq Z_i\leq Z_i''$ (the existence of the limits is guaranteed by Proposition \ref{proposition convergence}).
	The same argument as in the proof of Proposition \ref{proposition convergence} implies that  the variations of $A'_{i,2t+\iota}$ and $A''_{i,2t+\iota}$ are proportional to the respective  values of the Lyapunov functions $L_i'(T)$ and $L_i''(T)$. Consequently,  $Z'_{i}$ and $Z''_{i}$ are close to $(A'_{i,T+\iota})_{\iota=0,1}$ and $(A''_{i,T+\iota})_{\iota=0,1}$, respectively. Since $(A_{i,T+\iota})_{\iota=0,1}$ is close to $(\mu,\mu)$ with probability close to one, so are $(A'_{i,T+\iota})_{\iota=0,1}$ and $(A''_{i,T+\iota})_{\iota=0,1}$, and hence $Z'_i$ and $Z''_i$ are close to $(\mu,\mu)$; therefore so is $Z_i$, as desired.

	\subsection{Distorted monitoring }
	
	The above proof of Theorem~\ref{thm R}  relies on the three key properties of $\tfrac 1m$-DeGroot dynamics: (i) monotonicity, (ii) approximate averaging, (iii) robustness (Inequality~\eqref{eq gamma eta}). In fact, the arguments of the proof show that any updating rule that satisfies these properties induces approximate learning (the formal statement appears in Theorem~\ref{theorem mu}). An example of a different dynamics that fits into our framework is the \emph{status quo-biased DeGroot} dynamics discussed in Section~\ref{section status quo}.
	
	As mentioned in the introduction, the $\tfrac 1m$-DeGroot updating rule eliminates $\beta$-distorted monitoring for $\beta<\tfrac{1}{2m}$ and therefore it is immune to such distortion. Although not all updating rules that satisfy (i)--(iii) eliminate the distortion, they are still immune to $\beta$-distorted monitoring\footnote{See Theorem~\ref{theorem mu}.} for small enough $\beta$, as explained next.  
	
	First note that any updating rule that satisfies (i)--(iii) without distortion still satisfies (ii) and (iii) with a small enough distortion. However, it need not  satisfy (i). Therefore, the proof sketched above does not apply to it directly. To see why monotonicity may fail with the presence of distorted monitoring, consider two histories $A_{<t}$ and $B_{<t}$ such that $A_{i,s}$ is slightly smaller than $B_{i,s}$ for all $i\in V$ and $s<t$. An agent who suffers distorted monitoring may still view $A_{i,s}$ as larger than $B_{i,s}$, and therefore, even if her updating rule would be monotonic without the distortion it is no longer monotonic with the distortion.
	
	To overcome this difficulty, we consider two special cases of $\beta$-distorted monitoring that do preserve monotonicity. These special cases are the two extremes of distorted monitoring: the positive and the negative $\beta$-biased monitoring. Under positive $\beta$-biased monitoring agents overestimate the opinions of their neighbors; namely, they perceive them as $\beta$ higher than their actual value. Similarly, negative $\beta$-biased agents underestimate the opinions of their neighbors by $\beta$.
	
	Let $A_{i,t}$ be a dynamics induced by an updating rule satisfying (i)--(iii) with $\beta$-distorted monitoring on a {network} that satisfies the conditions of Theorem~\ref{thm R}. We couple $A_{i,t}$ with the negative and the positive $\beta$-biased dynamics,  $A^-_{i,t}$ and $A^+_{i,t}$, by sharing the same initial opinions. By the monotonicity property, we have $A^-_{i,t}\leq A_{i,t}\leq A^+_{i,t}$, for every agent $i$ and time $t\geq 0$. Since $A^-_{i,t}$ and $A^+_{i,t}$ satisfy (i)--(iii) they both satisfy approximate learning, and therefore, so does $A_{i,t}$.

	\section{General Framework}\label{section general framework}
	We present a general framework within which our main results hold. We provide sufficient conditions for a dynamics to satisfy $(\delta,\rho)$-learning. In addition, the underlying network may be a weighted graph rather than just a  graph.\footnote{A weighted graph is a graph with non-negative weights on its edges and possibly with self-loops.} 
	$\frac 1m$-DeGroot dynamics will turn out to be a special case of a more general family of dynamics on a weighted graph. We first provide a concentrated list of notations and other conventions that will be used in the statement of the general results, and subsequently during their proofs.
	
	\subsection{List of notations}
	
	\begin{itemize}
		
		\item The social network is a connected weighted graph denoted by $(G,W)$, where $G=(V,E)$ is a connected  graph, and  $W\colon V\times V\to \mathbb R_+$ is such that
		\begin{itemize}
			\item $W(i,j)=W(j,i)$, for all $i,j\in V$,
			\item $W(i,j)>0$ $\Rightarrow$ $ij\in E$ or $i=j$.
		\end{itemize}
		
		\item We assume that the weights of the edges in $E$ are bounded away from zero, and by normalization $W(i,j)\geq 1$, for all $ij\in E$.
		
		\item We denote $W(i):=\sum_{j\in V} W(i,j)$ and assume that $\sup_{i\in V} W(i)<\infty$.
		
		\item The neighbors of agent $i$ form the set $N_i$ and $(i):=\{i\}\cup N_i$.
		
		\item The opinion of agent $i\in V$ at time $t\geq 0$ is denoted by $A_{i,t}$.
		
		\item By conditioning on the state of the world $\mu$ and re-scaling the opinions, we assume w.l.o.g. that $A_{i,t}\in [0,1]$, for all $i\in V$ and $t\geq 0$. Unless otherwise specified, we assume that $\{A_{i,0}\}_{i\in V}$ are i.i.d.\ with expectation $\mu$. 
		
		\item The set of all possible histories is denoted by $\mathcal H$. Formally, $\mathcal H$ is defined as the set of all functions $A\colon V\times\mathbb Z_+\to [0,1]$.
		
		\item Each agent $i$ sees only her own opinion and those of her neighbors. For a history $A\in\mathcal H$ and agent $i$ and time $t>0$, the restriction of $A$ to $(i)\times \{0,\ldots,t-1\}$ is denoted by $A_{(i),<t}$. The set of all \emph{private histories of agent $i$ up to time $t$} is denoted by $\mathcal H_{(i),<t}:=\{A_{(i),<t}:A\in \mathcal H\}$.
		
		\item At time $t>0$, each agent $i$ updates her opinion according to an \emph{updating function} $g_{i,t}\colon \mathcal H_{(i),<t}\to [0,1]$. Namely, $A_{i,t}=g_{i,t}(A_{(i),<t})$. An \emph{updating rule} is a scheme of updating functions $\{g_{i,t}\}_{i\in V,t>0}$.
		
		\item[] We consider two specific updating rules defined next.
		\begin{itemize}
			\item The \emph{DeGroot} updating rule, denoted by $g_D\colon\mathcal H_{(i),<t}\to[0,1]$, is defined by
			\[
			g_D(A_{(i),<t})=g_D(A_{(i),t-1}):=\frac{1}{W(i)}\sum_{j\in V} W(i,j)A_{j,t-1}\ .
			\]
			
		\end{itemize}
		
		Let $M$ be a finite subset of $[0,1]$. The \emph{$M$-granular DeGroot} ($M$-DeGroot for short) updating function of agent $i$ at time $t>0$ is defined by first projecting the opinions of the neighbors on $M$, then applying the DeGroot updating function, and finally projecting the result on $M$. Formally, 
		define the projection function  $\pi\colon M\times[0,1]\to M$ by setting $\pi({\mathsf {m^*}},x)$ to be the closest element to $x$ in $M$, where ties between two alternatives are broken in favor of the one that is closer to ${\mathsf {m^*}}$. With a slight abuse of notation we extend $\pi$ to a function $\pi\colon M\times [0,1]^d\to M^d$ by applying it coordinate by coordinate.
		The $M$-granular DeGroot updating rule is defined by
		\[
		A_{i,t}=\pi(A_{i,t-2},g_D(\pi({\mathsf {m^*}},A_{(i),t-1}))),
		\]
		where ${\mathsf {m^*}}\in M$ is an arbitrary monotonic function of $A_{(i),<t}$. For example, ${\mathsf {m^*}}$ can be a constant element of $M$.
		
		Since the $M$-granular DeGroot updating function takes values in $M$, it always holds that $A_{(i),t-1}\in M^{(i)}$ and, therefore,  $\pi({\mathsf {m^*}},A_{(i),t-1})=A_{(i),t-1}$, regardless of the choice of ${\mathsf {m^*}}$. The purpose of this projection is to eliminate $\beta$-distorted monitoring, for any $\beta$ sufficiently small. The choice of the tiebreaker ${\mathsf {m^*}}$ is immaterial, since for $\beta$ small enough ties never occur here.
		
		Note that the tiebreaker in the outer projection is $A_{i,t-2}$ rather than $A_{i,t-1}$. The reason for that is the property of robustness that will be discussed in this section which is related to Inequality~\eqref{eq gamma eta}. 
		\begin{itemize}
			\item The $\frac 1m$-DeGroot updating function as defined in Definition~\ref{definition 1m DeGroot} can be alternatively defined with the above notation as
			\[
			A_{i,t}=\pi(A_{i,t-1},g_D(\pi(A_{i,t-1},A_{(i),t-1})).
			\]
		\end{itemize}
		$\frac 1 m$-DeGroot dynamics is similar to the special case of the $M$-granular DeGroot dynamics where \[M=\{0,\tfrac 1 m,\tfrac 2 m,\ldots,1\},\ 
		{\mathsf {m^*}}=A_{i,t-1},
		\] except that the second tiebreaker is $A_{i,t-1}$ instead of $A_{i,t-2}$. In order to bridge this difference, we turn the original underlying graph into a weighted graph {with additional edges between each vertex to itself, a.k.a.\ self-loops}. The weights of the weighted graph are 1 on the edges of the original graph and a small $a>0$ on all self-{loops}. The effect of this modification, for any $a>0$ small enough, is that the self-edge breaks potential ties in favor of $A_{i,t-1}$ instead of the formal tiebreaker, $A_{i,t-2}$. Therefore, the $M$-granular updating {rule} on the modified weighted graph induces the same dynamics as the $\frac 1m$-DeGroot updating {rule} on the original graph.
	\end{itemize}
	
	\subsection{Generalized results}
	
	The general framework encompasses any updating rule that satisfies the following three properties.
	\begin{definition}
		Let $\varepsilon,\gamma,\eta>0$, {$i\in V$, and $t>0$}. An updating {function}  $g_{i,t}\colon\mathcal H_{(i),<t}\to[0,1]$  is \emph{$(\varepsilon,\gamma,\eta)$-robust} if the following conditions are satisfied.
		
		\begin{enumerate}
			\item[(A1)] \emph{Monotonicity.}
			For any $A_{(i),<t},A_{(i),<t}'\in\mathcal H_{(i),<t}$ such that $\forall j,~\forall s<t,~A_{j,s}\leq A'_{j,s}$,
			\[
			g_{i,t}(A_{(i),<t})\leq g_{i,t}(A'_{(i),<t}).
			\]
			\item[(A2)] \emph{Approximate averaging}. For any $A_{(i),<t}\in\mathcal H_{(i),<t}$,
			\[
			|g_{i,t}(A_{(i),<t})-g_D(A_{(i),<t})|\leq \varepsilon.
			\]
			\item[(A3)] \emph{Robustness}. $\forall A_{(i),<t}\in\mathcal H_{(i),<t}$, let 
			\begin{align*}
				g_D&=g_D(A_{(i),<t}),\\
				g_{i,t}&= g_{i,t}(A_{(i),<t}).
			\end{align*}
			The third condition states that $\forall v\in [g_D-\gamma, g_D+\gamma]$,
			\[
			\eta|g_{i,t}-A_{i,t-2}|\leq \left(A_{i,t-2}-v\right)^2 - \left(g_{i,t}- v\right)^2.
			\]
		\end{enumerate}
	\end{definition}
	
	{
		Condition (A3) ensures that the variation of the opinion at any update is bounded from above by the decrease in the energy. This property plays a key role in the convergence of the opinions. Further motivation is given around Inequality~\eqref{eq gamma eta}.
		
		The robustness notion naturally extends from updating functions to updating rules (schemes of updating functions) and to the dynamics they induce. The following definition spells out this extension explicitly.
		\begin{definition}
			An updating rule $\{g_{i,t}\}_{i\in V,t>0}$ is $(\varepsilon,\gamma,\eta)$-robust if $g_{i,t}$ is $(\epsilon,\gamma,\eta)$-robust, for every $i\in V$ and $t>0$.
			
			A dynamics induced by an $(\varepsilon,\gamma,\eta)$-robust updating rule is called an \emph{$(\varepsilon,\gamma,\eta)$-robust dynamics}. 
			
			A dynamics induced by an $(\varepsilon,\gamma,\eta)$-robust updating rule and a set of stubborn agents and/or $\beta$-distorted monitoring is called an \emph{$(\varepsilon,\gamma,\eta)$-robust dynamics with stubborn agents and/or $\beta$-distorted monitoring}.
		\end{definition}
		
		Note that the standard DeGroot updating rule satisfies Conditions (A1) and (A2) but not (A3).
	}
	Also, note that Conditions (A2) and (A3) can only be satisfied if $\gamma\leq \varepsilon$.
	
	Condition (A3) does not impose any condition on the initial opinions $\{A_{i,0}\}_{i\in V}$. Furthermore, the predicate of Condition (A3) is defined on any history in $\mathcal H$, not just those induced by some updating rule. 
	{
		
		In the general framework, it is necessary to consider such inconsistent histories since such histories that never occur may still be perceived due to distorted monitoring. 
		
		In the special case of $\frac 1 m$-DeGroot updating rule, Condition (A3) is \emph{not} satisfies as is. However, for $\frac 1 m$-DeGroot, it is sufficient to consider consistent histories only, since $\frac 1 m$-DeGroot eliminates distorted monitoring. To this end, we introduce a weakening of the robustness definition.
		
		\begin{definition}
			\label{definition: weak robustness}
			Let $g=\{g_{i,t}\}_{i\in V,t>0}$ be an updating rule and $\mathcal H^g$ the set of histories consistent with $g$. We say that $g$ is $(\varepsilon,\gamma,\eta)$-weakly robust if, for every $i\in V$ and $t>0$, $g_{i,t}$ satisfies \begin{enumerate}
				\item[(A1)] \emph{Monotonicity.}
				\item[(A2)] \emph{Approximate averaging}. 
				\item[(A3')] \emph{Weak robustness}. $\forall A_{(i),<t}\in\mathcal H^g_{(i),<t}$, let 
				\begin{align*}
					g_D&=g_D(A_{(i),<t}),\\
					g_{i,t}&= g_{i,t}(A_{(i),<t}).
				\end{align*}
				Then, $\forall v\in [g_D-\gamma, g_D+\gamma]$,
				\[
				\eta|g_{i,t}-A_{i,t-2}|\leq \left(A_{i,t-2}-v\right)^2 - \left(g_{i,t}- v\right)^2.
				\]
			\end{enumerate}
		\end{definition}
	}
	
	\begin{definition}
		For $i\in V$, we define $Z_i\in[0,1]\times[0,1]$ as long as the following limit exists, 
		\[
		Z_i:=(\lim_{t\to\infty} A_{i,2t+\iota})_{\iota=0,1}.
		\]
	\end{definition}
	
	The following theorem says that $Z_i$ is well defined under Condition (A3).
	\begin{theorem}\label{theorem convergence}
		% {\color{blue}
			% Let $(G=(V,E),W)$ be a weighted graph with sub-exponential growth rate and bounded weights ( $\sup_{i\in V}W(i)<\infty$). Let $g=\{g_{i,t}\}_{i\in V,t>0}$ be an updating rule satisfying condition (A3) w.r.t.\ some $\gamma >\beta$ and $\eta>0$; and $A=\{A_{i,t}\}_{i\in V,t\geq 0}$ a dynamics on $G$ 
			%  induced by $g$ possibly with stubborn agents and $\beta$-distorted monitoring.
			
			%  Then the limits $Z_i$ exist for every $i\in V$.
			% }
		{
			Let $(G=(V,E),W)$ be a weighted graph, $g=\{g_{i,t}\}_{i\in V,t>0}$  an updating rule, and $A=\{A_{i,t}\}_{i\in V,t\geq 0}$ a dynamics on $G$ 
			induced by $g$ with any set of stubborn agents and with $\beta$-distorted monitoring. The conjunction of the following conditions implies that all the limits $Z_i$ exist, for every $i\in V$.
			\begin{itemize}
				\item $G$ has a sub-exponential growth rate;
				\item $\sup_{i\in V}W(i)<\infty$;
				\item Either $\beta\geq 0$ and $g$ satisfies Condition (A3) or $\beta  = 0 $ and $g$ satisfies Condition (A3') w.r.t.\ some $\gamma >\beta$ and $\eta>0$.
			\end{itemize}
		}
	\end{theorem}
	{
		Note that Theorem~\ref{theorem convergence} does not make  any assumption regarding the initial opinions $\{A_{i,0}\}_{i\in V}$ which can be arbitrary numbers in $[0,1]$.
	} The proof of Theorem~\ref{theorem convergence} is in {Appendix}~\ref{section proof theorem conv}.
	
	The next theorem sets conditions under which $Z_i$ is close to $(\mu,\mu)$ with a probability close to one.
	{
		
		Let $G=(V,E)$ be a graph. For $i,j\in V$ we write $\mathrm{dis}_G(i,j)$ for the graph distance between $i$ and $j$ in $G$. The diameter of $G$ is defined as $\mathrm{dim}(G):=\sup_{i,j\in V} \mathrm{dis}_G(i,j)$. For $S\subset V$ we define 
		\[
		\mathrm{dis}_G(i,S):=\begin{cases}
			\min_{j\in S}{\mathrm{dis}_G(i,j)} & S\neq\emptyset,\\
			\mathrm{dim}(G)+1 & S=\emptyset.
		\end{cases}
		\]
		Where, $\mathrm{dis}_G(i,S):=\infty$, if $S$ is empty and $G$ is infinite.
		
		\begin{theorem}\label{theorem mu}
			Let $\overline{W},\varepsilon,\gamma,\eta,\delta>0$, $\beta\in[0,\gamma)$ and stretched sub-exponential function $f\colon\mathbb R_+\to\mathbb R_+$. There exist positive numbers $R=R(\varepsilon,\gamma,\eta,\overline{W},\delta,f,\beta)$, $\rho=\rho(\varepsilon,\gamma,\eta,\overline{W},\delta,f,\beta)$, such that any $(\varepsilon,\gamma,\eta)$-robust   $\beta$-distorted monitoring dynamics on a weighted graph $(G=(V,E),W)$ such that $\sup_{i\in V} W(i)\leq\overline W$ and $G$ is majorized by $f$ and any (possibly empty) set $S\subset V$ of stubborn agents satisfies the following: 
			
			If the initial opinions $\{A_{i,0}\}_{i\in V}$ are i.i.d.\ with expectation $\mu$, then for every $i\in V$ for which $\mathrm{dis}_G(i,S)>R$ \[\Pr(\Vert Z_i-(\mu,\mu)\Vert >\delta)\leq\rho.\]
			Furthermore, for fixed $\overline{W}$, $\delta$, and $f$, and any $K>1$,
			\[
			\sup_{\substack{\eta,\gamma,\beta\,:\\\eta,(\gamma-\beta)\geq \varepsilon^{K}}}\rho(\varepsilon,\gamma,\eta,\overline{W},\delta,f,\beta)\xrightarrow[\varepsilon\to 0^+]{}0.
			\]
			Moreover, for $\beta=0$, weak robustness suffices. 
		\end{theorem}
	}
	
	The proof of Theorem~\ref{theorem mu} is in {Appendix}~\ref{section proof thm mu}. 
	
	\section{Main Results Fit into General Framework}\label{section 1m}
	
	In this section, we  show that  Theorems~\ref{theorem convergence} and \ref{theorem mu} imply Theorems~\ref{theorem: main one graph} and \ref{theorem: main family}. 
	
	Recall that the $\tfrac 1m$-DeGroot updating rule can be implemented through the $\frac 1m$-grid granular DeGroot updating rule on a certain weighted graph. We associate a graph $G=(V,E)$ and a number $a>0$ with a weighted graph $(G,W=W_{G,a})$ given by 
	
	\[W(i,j)=\begin{cases}
		1& (i,j)\in E,\\
		a& i=j,\\
		0& \text{otherwise}.
	\end{cases} \]
	
	The following lemma is a crucial step in reducing the main results to the general framework.
	
	\begin{lemma}\label{lemma 1m robustness}
		For any $m$ and $d$ there exists $a\in(0,1)$ such that for any graph of degree at most $d$
		the $\tfrac 1m$-grid granular DeGroot updating rule on the weighted graph $(G, W_{G,a})$  is $(\varepsilon, \gamma, \eta)$-{weakly} robust, where $\varepsilon=\tfrac{1}{2m}$ and $\gamma,\eta = \tfrac{1}{4md^2}$.  
	\end{lemma}
	
	The proof of Lemma~\ref{lemma 1m robustness} appears at the end of this section.

	\begin{proof}[Theorems~\ref{theorem convergence} and \ref{theorem mu} imply Theorem~\ref{thm R}]
		
		Consider the $\tfrac 1m$-DeGroot updating function on the graph $G$. 
		For $a\in(0,1)$ small enough, the $\tfrac 1m$-DeGroot and the $\tfrac 1m$-grid granular DeGroot updating rules on $G$ and $(G,W_{G,a})$, respectively, induce the same dynamics $A_{i,t}$. By Lemma~\ref{lemma 1m robustness}, the $\tfrac 1m$-grid granular DeGroot updating rule is $\left(\tfrac 1{2m},\tfrac 1{4md^2}, \tfrac 1{4md^2}\right)$-{weakly} robust. Therefore, by Theorem~\ref{theorem convergence}, the limit $Z_i=(\lim_{t\to\infty} A_{i,2t+\iota})_{\iota=0,1}$ exists. 
		
		By Theorem~\ref{theorem mu}, applied with $\overline{W}=d+1$, $\varepsilon=\tfrac 1{2m}$, $\gamma=\eta=\tfrac{1}{{4md^2}}$, $\delta>0$ and $\beta=0$, there exist $R(m,d,f,\delta)$ and $\rho(m,d,f,\delta)$ such that every agent $i\in V$ who is at least $R(m,d,f,\delta)$-away from any stubborn agent satisfies $(\delta,\rho)$-learning. Furthermore, since $\gamma$ and $\eta$ are (linear) polynomials in $\tfrac 1m$ we have  $\lim_{m\to\infty}\rho(m,d,f,\delta)=0$. Therefore, given $\rho>0$, one can take $m$ large enough such that any agent who is $R(m,d,f,\delta)$-away from any stubborn agent satisfies $(\delta,\rho)$-learning, which completes the proof.
		
	\end{proof}
	
	We prepare for the proof of Lemma~\ref{lemma 1m robustness}. To this end, we consider the $M$-granular DeGroot updating rule. Let $M\subset[0,1]$ be a finite set. For $x\in[0,1]$ and $r>0$, let $B(x,r):=[x-r,x+r]\cap [0,1]$.
	Define \[\varepsilon=\varepsilon(M):=\inf\left\{r>0: \cup_{m\in M}B(m,r)=[0,1]\right\},\quad \text{and}\] 
	\[\rho=\rho(M):= \sup\left\{r>0 :  B(m,r)\cap B(m',r)=\emptyset,~ \forall m,m'\in M,~ m\neq m' \right\}.\]
	
	Let $d\geq 2$ and define \[M^d:=M\cup \frac{M+M}{2} \cup \frac{M+M+M}{3}\cup\cdots\cup\frac{M+M+\cdots+M}{d}.\]
	
	\begin{lemma}\label{lemma granular}
		Let $M\subset[0,1]$ be a finite set, $G=(V,E)$ a graph of degree at most $d$, and $a\in(0,\rho(M^d))$. Let $\varepsilon= \rho(M)$ and $\gamma=\eta=\rho(M^d)-a$. For every $i\in V$ and $t>0$, the $M$-granular DeGroot updating function on $(G,W_{G,a})$, $g_{i,t}$, satisfies Conditions
		(A1) and (A2) if $A_{i,t-2}\in M$, and satisfies (A3) if in addition $A_{j,t-1}\in M$ for every $j\in (i)$.
	\end{lemma}
	
	\begin{proof}
		We show that $g_{i,t}$ is
		\begin{enumerate}[(i)]
			\item monotonic,
			\item $2\varepsilon$-averaging in general and $\varepsilon$-averaging if $A_{(i),t-1}\in M^{(i)}$, and
			\item $(\gamma,\eta)$-robust if $A_{(i),t-1}\in M^{(i)}$,
		\end{enumerate}
		for every $i\in V$, and $t>0$.

		Condition (i) holds, since $g_{i,t}$ is defined as a composite of monotone functions. Condition (ii) holds since the projection $\pi(m,x)$ is within $\varepsilon$ of $x$ for every $m\in M$ and $x\in[0,1]$, and $\pi(m,x)=x$ if $x\in M$. Therefore, $g_D(\pi(m,A_{(i),t-1}))$ is within $\varepsilon$ of $g_D:=g_D(A_{(i),t-1})$, and $g_D(\pi(m,A_{(i),t-1}))=g_D$ if $A_{(i),t-1}\in M^{(i)}$. The second projection in the definition of $g_{i,t}$ changes the result by at most an additional $\varepsilon$.
		
		It remains to verify Condition (iii). 
		Since we assume that $A_{(i),t-1}\in M^{(i)}$, we have $g_{i,t}:=g_{i,t}(A_{(i),<t})=\pi(A_{i,t-2},g_D)$. Let $v\in B(g_D,\gamma)$. Define   $x:=A_{i,t-2}-v$ and $y:= g_{i,t}-v$. Then condition (iii) can be expressed as
		\[
		\eta|x-y|\leq x^2-y^2.
		\]
		
		We first claim that either $x=y$, or $|x|-|y|\geq \eta$. To see that, note that since $g_{i,t}$ is a closest point to $g_D$ in $M$ and $A_{i,t-2}\in M$, we must have either $x=y$, or
		\begin{align*}
			A_{i,t-2}&&<&& \frac{A_{i,t-2}+g_{i,t}} 2 &&<&&  g_D, g_{i,t},&&\text{or}\\
			A_{i,t-2}&&>&& \frac{A_{i,t-2}+g_{i,t}} 2 &&>&&  g_D, g_{i,t}.&&
		\end{align*}
		Since all these points are in $M^d+(-a,a)$, they are either equal or $2\eta$ away one from the next one. The claim follows, since $|v-g_D|\leq \gamma=\eta$. 
		
		Now, if $x=y$, Condition (iii) holds trivially. Otherwise, $|x|-|y|\geq \eta$, and then
		\[
		\eta|x-y|\leq \eta(|x|+|y|)\leq (|x|-|y|)(|x|+|y|)=x^2-y^2.
		\]
	\end{proof}
	
	We are now ready to prove Lemma~\ref{lemma 1m robustness}. 
	
	\begin{proof}[Proof of Lemma~\ref{lemma 1m robustness}:] Let $M$ be the $\tfrac{1}{m}$-grid, and $a=\tfrac{\rho(M^d)}{2}$. From Lemma~\ref{lemma granular}, we have that the $\tfrac 1m$-grid granular DeGroot updating rule on the weighted graph $(G, W_{G,a})$  is $(\varepsilon, \gamma, \eta)$-robust, where $\varepsilon=\varepsilon(M)=\tfrac{1}{2m}$ and $\gamma,\eta = \tfrac{\rho(M^d)}{2}\geq\tfrac{1}{4md^2}$. 
		
	\end{proof}

	\section{Status quo-biased DeGroot dynamics}\label{section status quo}
	In this section we present a different heuristic updating rule that  satisfies the three Conditions (A1)--(A3), and, therefore, leads to robust approximate learning. At time $t>1$, each agent $i$ computes the DeGroot value $g_D(A_{(i),<t})$ and then shifts it slightly toward its own previous opinion, $A_{i,t-2}$ (the status quo).  
	
	Formally, we define a family of updating rules parameterized by $\varepsilon>0$ called 
	the $\varepsilon$-Status Quo-Biased DeGroot ($\varepsilon$-SQBD) updating rule. The  $\varepsilon$-SQBD updating function for agent $i$ at time $t$, $g_{i,t}\colon\mathcal{H}_{(i),<t}\to [0,1]$, is defined as follows. Let $A_{(i),<t}\in\mathcal{H}_{(i),<t}$ and  $g_D:=g_D(A_{(i),<t})=\frac 1 {W(i)}\sum_{j\in V}W(i,j)A_{j,t-1}$.  Then,
	
	\[
	g_{i,t}(A_{(i),<t}):=\left\{
	\begin{array}{llc}
		A_{i,t-2} & |A_{i,t-2}-g_D|\leq \varepsilon,\\
		g_D+\varepsilon & A_{i,t-2}>g_D+\varepsilon,\\
		g_D-\varepsilon & A_{i,t-2}<g_D-\varepsilon.
	\end{array}
	\right.
	\]
	
	\begin{lemma}\label{lemma approximate}
		The $\varepsilon$-SQBD updating function is $(\varepsilon,\gamma,2(\varepsilon-\gamma))$-robust for any $\varepsilon>\gamma>0$ .
	\end{lemma}
	
	\begin{proof}
		We show that the $\varepsilon$-SQBD updating function satisfies (A1), (A2), and (A3).
		
		The $\varepsilon$-SQBD updating function is monotonic (satisfies (A1)) because it can be presented as the following composite of monotone functions:
		
		\[g_{i,t}(A_{(i),<t})=\min \left\{\max\left\{d_D-\varepsilon,A_{i,t-2}\right\},g_D+\varepsilon\right\}.\]
		
		The $\varepsilon$-SQBD updating function satisfies (A2) because each one of the three cases detailed in the definition satisfies (A2).
		
		Lastly, let $\varepsilon>\gamma>0$ and $\eta=2(\varepsilon-\gamma)$. We show that condition (A3) is satisfied in each one of the cases of the definition.
		In the case that $|A_{i,t-2}-g_D|\leq \varepsilon$, (A3) is satisfied trivially.
		The other two cases are symmetric. We consider the case $A_{i,t-2}>g_D+\varepsilon$. Let $v\in[g_D-\gamma,g_D+\gamma]$. Define $x:=A_{i,t-2}-v$ and $y:= g_D+\varepsilon-v$. Condition (A3) can be expressed as \[\eta (x-y) \leq x^2 - y^2.\] 
		Since $x>y$, the above is equivalent to $\eta\leq x+y$. Indeed,
		\[x+y\geq 2y = 2\left(g_D+\varepsilon-v\right)\geq 2\left(g_D+\varepsilon-(g_D+\gamma)\right)=\eta.\]
		
	\end{proof}
	
	The following corollary is an immediate consequence of Theorems~\ref{theorem convergence} and \ref{theorem mu}.
	\begin{corollary}
		Let $f\colon\mathbb R_+\to\mathbb R_+$ be a stretched sub-exponential function, $\overline W>0$, and $\delta,\rho>0$. There exist $\varepsilon>0$ and $R>0$
		such that the $\varepsilon$-SQBD dynamics on any weighted connected graph $(G=(V,E),W)$ such that $G$ is majorized by $f$, $\sup_{i\in V}W(i)\leq \overline W$, and a radius of $G$ is at least $R$ satisfies the following:
		\begin{enumerate}[(i)]
			\item every agent $i$ such that $B(i,R)$ does not include a stubborn agent satisfies $(\delta,\rho)$-learning;
			\item $(\delta,\rho)$-learning continues to hold even when all agents suffer from $0.9\varepsilon$-distorted monitoring.
		\end{enumerate}
	\end{corollary}
	\section{Discussion}\label{section discussion}
	Our paper provides an added value to the existing literature by considering a robust version of the classic DeGroot dynamics one that satisfies Golub and Jackson \cite{golub2010naive} wisdom of crowds result even in the presence of {adversarial} agents. Unlike Golub and Jackson \cite{golub2010naive} our theorems are formulated to both finite and infinite graphs. In Theorem \ref{theorem:golubjackson}, we generalize Golub and Jackson's result to infinite graphs rather than a sequence of growing graphs. {To the best of our knowledge, our work is the first time the DeGroot and majority models are generalized by one common model, the $\frac 1 m$-DeGroot. 
		
		We utilize techniques from the literature of both models while non-trivially adapting these techniques to our model. The convergence and concentration results from \cite{degroot1974reaching} and \cite{golub2010naive} respectively are both utilized while being generalized and strengthened to accommodate infinite graphs and provide quantitative estimates of the convergence rate and concentration measure. The Lyapunov function of \cite{tamuz2015majority}, which is based on the energy function of \cite{ginosar2000majority}, is applied here quantitatively rather than just qualitatively as has been done previously.
		
		The proof of our main Theorems~\ref{theorem: main one graph}, \ref{theorem: main family} and \ref{thm R} has an economic insight. The effect of the granularity of opinions is that opinions change only until they become locally close to each other and after that, they remain fixed. Since averaging procedures result in local closeness at a finite time, opinions only have a finite time to spread in the network and therefore spread a finite distance only. The result of this is that the influence of false opinions is limited to a finite radius.
	}  
	
	Since our $\frac{1}{m}$-DeGroot is a parameter-dependent family, one may ask how the choice of $m$ affects the learning results. Interestingly, the choice of $m$ balances two opposing effects. On the one hand, as $m$ grows so does the accuracy of the agents' estimation of the state of the world. On the other hand, as $m$ grows so does the radius of influence of the {adversarial} agents.  Another thing that is worth a discussion is our choice of the tie-breaking rule for the $\frac 1 m $-DeGroot heuristics and whether this choice is realistic. To answer this, we refer to an interpretation of the tie-breaking rule that appears in the technical proof. One may replace the tie-breaking rule by assuming that each agent assigns a certain weight to her own opinion when computing the average of neighboring opinions. Generically, this prevents the possibility of ties. When one's own weight is small enough the resulting dynamics is identical to the one we proposed (with the tie-breaking rule).
	
	Regarding the existence of {adversarial} agents, our main results, Theorem~\ref{theorem: main one graph} and Theorem~\ref{theorem: main family}, state that if the number of {adversarial} agents is bounded, then their influence on the opinions of the agents in the network is limited. One may ask what would be the influence of an unbounded (infinite) set of {adversarial} agents. Closely looking at the proof, Theorem~\ref{thm R} states that the influence of each {adversarial} agent is limited to a certain radius around it. This implies that even if the number of {adversarial} agents is unbounded, as long as their presence in the network is sparse enough, they still have a limited influence.  
	
	The connection between the graph structure of a given social network and the quality of learning is studied extensively in many contexts. Unfortunately, our results do not apply to infinite graphs of exponential growth rate (e.g., trees) and finite graphs of relatively small radius (e.g., small world networks). Extending our results to such networks is an interesting research direction. In fact, there are many open questions in this direction even in the case of $m=1$, namely, the Majority dynamics. Benjamini et al.\  conjectured that the majority dynamics with i.i.d.\ initial opinions (alternately) converges on any bounded degree infinite graph \cite[Conjecture 1.3]{benjamini2016convergence}. Tamuz and Tessler  \cite{tamuz2015majority} made several conjectures on the possibility to tell the state of the world out of the collective limit opinions of all the agents in the network.

	{} Important questions for future research are whether our results extend in the following {directions}: (1) each agent $i$ applies the $\frac 1 m_i$-DeGroot updating rule with different $m_i$s? Specifically, in the case where all $m_i$s are bounded from above and from below by global bounds; (2) instead of simultaneous updating the agents update asynchronously according to independent Poisson clocks; {(3) graphs of faster growth rates, e.g., regular trees.
		
		Another direction for future research is to derive tighter bounds on the radius of influence of adversarial agents.}

	\section{Acknowledgements}
	{We are grateful to the anonymous referee and associate editor for many insightful comments and in particular for suggesting considering adversarial agents beyond just stubborn ones.}
	
	G. Amir gratefully acknowledges the support of the Israeli Science Foundation grant \#957/20. I.
	Arieli and R. Peretz gratefully acknowledge the support of the Israel Science Foundation grant \#
	2566-20. G. Ashkenazi-Golan gratefully acknowledges the support of the Israel Science Foundation,
	grants \#217/17 and \#722/18, and NSFC-ISF, China Grant \#2510/17.
	% Bibliography
	\bibliographystyle{plain}
	\bibliography{bibl}
	%\bibliography{EC22/sample-bibliography}
	
	% Appendix
	\appendix
	\section{Proof of Theorem~\ref{theorem convergence}}\label{section proof theorem conv}
	
	Theorem~\ref{theorem convergence} is purely combinatorial. There is no probability involved. The results hold for any initial opinions $A_{i,0}$, and Condition (A3) can be assumed with respect to any $\gamma\in(0,1)$ and $\eta>0$. The mathematical argumentation follows the work of by Durrett and Steif \cite{durrett1993fixation}, Ginosar and Holzman   \cite{ginosar2000majority}, and Tamuz and Tessler \cite{tamuz2015majority} on the majority dynamics. Since majority dynamics is a special case of $\frac 1 m$-DeGroot dynamics, it  satisfies Condition (A3). Therefore, Theorem~\ref{theorem convergence} generalizes the convergence results of the above-mentioned papers.

	The proof of Theorem~\ref{theorem convergence} follows from the conjunction of Proposition~\ref{proposition convergence} below and Lemma~\ref{lemma stubborn and bias reduction} that follows.
	\begin{proposition}\label{proposition convergence}
		Any bounded robust dynamics on a weighted graph of sub-exponential growth rate alternately converges.
	\end{proposition}
	Later in this section Proposition~\ref{proposition convergence} is obtained as an immediate corollary of a stronger statement, Lemma~\ref{lemma variation}.
	\begin{lemma}\label{lemma stubborn and bias reduction}
		Let $g$ be an updating rule that satisfies Condition (A3) w.r.t.\ $\gamma,\eta>0$ and let $\beta\in[0,\gamma)$. Then, the dynamics induced by $g$ with $\beta$-distorted monitoring and either with or without stubborn agents is $(\gamma-\beta,\eta)$-robust.
	\end{lemma}
	
	\begin{proof}[Proof of Lemma~\ref{lemma stubborn and bias reduction}]
		Let $g$ be an updating rule that satisfies Condition (A3) w.r.t. $\eta,\gamma>0$, and let $\beta\in[0,\gamma)$. We may assume w.l.o.g.\ that there are no stubborn agents. Indeed, if there are stubborn agents, say $i_0\in V$ is stubborn, we can modify $g_{i_0,t}$ so that $g_{i_0,t}(A_{(i_0),<t}):=A_{i_0,0}$, for all $t\geq -1$ and no longer assume that $i_0$ is stubborn. The modified updating rule is $(\gamma,\eta)$-robust while inducing the same dynamics as if $i_0$ were stubborn.
		
		Define a predicate
		\[
		\mathcal P(x,y,v)\quad \Leftrightarrow \quad \eta|y-x|\leq (y-v)^2-(x-v)^2.
		\]
		Let $A\in\mathcal H$ be the history induced by $g$ with $\beta$-distorted monitoring. Let $i\in V$ and $t>0$ be arbitrary. Denote $g_D:= g_D(A_{(i),<t})$. We have to show that
		\begin{align}\label{equation P}
			\begin{split}
				&\mathcal P(A_{i,t},A_{i,t-2},v),\\
				&\forall v\in[g_D-(\gamma-\beta),g_D+(\gamma-\beta)].
			\end{split}
		\end{align}
		
		Let $A'_{(i),<t}$ be the $\beta$-distorted version of $A_{(i),<t}$ such that $A_{i,t}=g_{i,t}(A'_{(i),<t})$. Denote  $g'_D:= g_D(A'_{(i),<t})$. Since  $A'_{i,t-2}=A_{i,t-2}$, 
		\[
		\mathcal P(A_{i,t},A_{i,t-2},v) \quad \Leftrightarrow \quad \mathcal P(g(A'_{(i),<t}),A'_{i,t-2},v),
		\]
		which holds for every $v\in[g'_D-\gamma,g' _D+\gamma]$, since 
		$g$ satisfies Condition (A3) w.r.t.\ $\eta$ and $\gamma$. 
		
		Since $|g'_D -g_D|\leq \beta$,
		\[
		[g_D-(\gamma-\beta),g_D+(\gamma-\beta)] \subset [g'_D-\gamma,g'_D+\gamma],
		\]
		which implies \eqref{equation P}.
	\end{proof}
	
	The main tool utilized in the proof of Proposition~\ref{proposition convergence} is a Lyapunov function, defined in the next couple of paragraphs. For a  weighted graph $(G,W)$, an agent $i\in V$, and a set of agents $e\subset  V$ we denote by $d(i,)$ the distance between $i$ and $e$ in the graph, namely, $d(i,e)=\min_{j\in e}d(i,j)$. For $r\in(0,1)$ we define:
	\[
	M_i(r)=\sum_{j,k\in V} W(j,k)r^{d(i,\{j,k\})}.
	\]
	The following lemma relates the growth rate of $G$ to the growth rate of $M_{i}(r)$ as a function of $1/(1-r)$.
	\begin{lemma}\label{lemma M_i}
		Let $(G,W)$ be a weighted graph with $\overline W:=\sup_{i}W(i)<\infty$. If $G=(V,E)$ has a sub-exponential growth rate, then $M_i(r)<\infty$, for every $r\in(0,1)$ and $i\in V$. Furthermore, for any stretched sub-exponential function $f\colon \mathbb R_+\to\mathbb R_+$ there exists a stretched sub-exponential function $g\colon \mathbb R_+\to\mathbb R_+$, such that if $G=(V,E)$ is majorized by $f$, then it holds that $\sup_{i\in V}M_i(r)\leq \overline W \cdot g(1/(1-r))$, for every $r\in(0,1)$.    
	\end{lemma}
	\begin{proof}
		Let $i\in V$ and $f\colon \mathbb R_+\to\mathbb R_+$ a non-decreasing function such that $|B(i,\cdot)|\leq f(\cdot)$. We think of $f$ as a sub-exponential function in the first part of the lemma and a stretched sub-exponential function in the second part of the lemma.
		Let $r\in(0,1)$ and $k(r)=\lceil1/(1-r)\rceil$. Since $r^{k(r)}\leq 1/e$, we have
		\begin{multline}\label{equation M(r)}
			M_i(r)= \sum_{j\in V}\sum_{k\in(j)}W(j,k)r^{d(i,\{j,k\})}\leq \sum_{j\in V}W(j)r^{[d(i,j)-1]_+}\leq \overline W \sum_{j\in V}r^{[d(i,j)-1]_+}\\
			< \overline W |B(i,k(r))| + \frac{1}{e} \overline W  |B(i,2k(r))|+\ldots +\frac{1}{e^{n-1}}\overline W |B(i,nk(r))|+\ldots\\
			\leq \overline W e\left(\sum_{n=1}^\infty \frac{f(kn)}{e^n}\right).
		\end{multline}
		
		If $f$ is sub-exponential, then $f(kn)<e^{\frac 1 2 n}$, for every $n$ large enough. Therefore, the tail of the series on the right-hand side of \eqref{equation M(r)} is dominated by a convergent geometric series and, therefore, $M_i(r)<\infty$. This proves the first part of the lemma.
		
		To prove the second part of the lemma it is sufficient to prove that the right-hand side of \eqref{equation M(r)} is a stretched sub-exponential function of $k$, assuming that $f$ is stretched sub-exponential. Namely, it is sufficient to show that
		\[
		\sum_{n=1}^\infty \frac{f(kn)}{e^n}
		\]
		is a stretched sub-exponential function of $k$, whenever $f$ is  a non-decreasing stretched sub-exponential function.
		Indeed, since $f$ is non-decreasing and 
		\[
		kn\leq \frac{k^2+n^2}{2}\leq \max\{k^2,n^2\},
		\]
		we have \[f(kn)\leq f(\max\{k^2,n^2\})\leq f(k^2)+f(n^2).\] Hence, 
		\[
		\sum_{n=1}^\infty \frac{f(kn)}{e^n} \leq f(k^2)\sum_{n=1}^\infty\frac{1}{e^n} + \sum_{n=1}^\infty\frac{f(n^2)}{e^n}.
		\]
		This completes the proof of the second part of the lemma since the function $f(k^2)$, is stretched sub-exponential since it is composite of a polynomial and a stretched sub-exponential function, and $\sum_{n=1}^\infty\frac{f(n^2)}{e^n}$ is a finite constant since $f(n^2)$ is sub-exponential in $n$. 
	\end{proof}
	For every  three vertices $i,j,k\in V$, denote $w(i,j,k):=W(j,k)\cdot r^{d(i,\{j,k\})}$.
	Following Tamuz and Tessler~\cite{tamuz2015majority}, given a history $A\in\mathcal H$ and $r\in(0,1)$, we define a Lyapunov function.
	\begin{definition}\label{definition L}
		For $i\in V$, $r\in(0,1)$, and $t\in\mathbb Z_+$, we define
		\begin{align*}L_{i,r}(t)&:=\sum_{j,k\in V}w(i,j,k)\left(A_{j,t+1}-A_{k,t}\right)^2\\
			&=\sum_{j\in V}\sum_{k\in (j)}w(i,j,k)\left(A_{j,t+1}-A_{k,t}\right)^2\\
			&=\sum_{j\in V}\sum_{k\in (j)}w(i,j,k)\left(A_{j,t}-A_{k,t+1}\right)^2.\end{align*}
	\end{definition}
	
	Note that, since we assume that $A_{i,t}\in[0,1]$, we have $0\leq L_{i,r}(t)\leq M_i(r)$.
	For readability, we sometimes suppress $i$ and $r$ and write $L(t):=L_{i,r}(t)$. 
	
	We define
	
	\[J^+_{j}(t):=\sum_{k\in (j)}w(i,j,k)\left(A_{j,t+1}-A_{k,t}\right)^2\]

	\[J^-_{j}(t):=\sum_{k\in (j)}w(i,j,k)\left(A_{j,t-1}-A_{k,t}\right)^2\] 
	By considering the last two equivalent formulations of Definition~\ref{definition L}, we have
	\[
	L_{i,r}(t)=\sum_{j\in V} J^+_{j}(t)=\sum_{j\in V} J^-_{j}(t+1).
	\]
	It follows that
	\begin{equation}\label{equation L jumps}
		L_{i,r}(t)-L_{i,r}(t-1)=\sum_{j\in V}\left(J^+_{j}(t)-J^-_{j}(t)\right).
	\end{equation}
	
	\begin{lemma}\label{lemma J}
		For every  weighted graph $(G,W)$,   agent $i\in V$,  $\gamma\in(0,1)$,  $r\in[1-\gamma,1)$,  $\eta>0$, and  every $(\gamma,\eta)$-robust $A\in\mathcal H$, 
		\[
		J^+_j(t)-J^-_j(t)\leq 0,
		\]
		for all $j\in V$ and $t>0$.
	\end{lemma}
	
	\begin{proof} 
		Fix $j\in V$ and $t>0$. Let $f(x):=\sum_{k\in (j)}w(i,j,k)\left(x-A_{k,t}\right)^2$. The function $f(x)$ is a second-degree polynomial, whose minimum is attained at $v=\sum_{k\in (j)}q(k)A_{k,t} $, where $q(k):=\tfrac{w(i,j,k)}{\sum_{\ell\in (j)}w(i,(j, \ell))}.$ Therefore, $f(x)=C_1(x-v)^2+C_2$ for some $C_1>0$ and $C_2\in\mathbb R$. Let $g_D:={ \frac{1}{W(j)}\sum_{k\in(j)}W(j,k)A_{k,t}}$. Assumption (A3) implies that if $|g_D-v|<\gamma$, then $(A_{j, t+1}-v)^2-(A_{j, t-1}-v)^2\leq 0$. Therefore,  since $J^+_j(t)-J^-_j(t)=f(A_{j,t+1})-f(A_{j,t-1})$, it remains to prove that 
		$|g_D-v|<\gamma$. 
		
		Define $u\in\Delta(j)$ by $u(k)\sim W(j,k)$, $\forall k\in (j)$. Let $a(k):=A_{k,t}~\forall k\in (j)$. Note that $g_D=\langle u,a\rangle$ and $v=\langle q,a\rangle$.

		We show that $||u-q||_{TV}\leq \gamma$. Recall that

		\[
		||u-q||_{TV}=\max_{f\colon (j)\to[0,1]}\int f\,\mathrm d(u-q)=\sum_{\substack{k\in (j):\\u(k)>q(k)}} [u(k)-q(k)].
		\]
		There exists $D$, such that for every $k\in (j)$, we have
		$d(i,j,k)\in\{D,D+1\}$. Let 
		\[
		\alpha:= \sum_{\substack{k\in {(j)}:\\d(i,j,k)=D+1}}u(k).
		\]
		Then 
		
		\[
		q(k)=\begin{cases}
			\frac {u(k)} {(1-\alpha(1-r))}&\text{if $d(i,j,k)=D$, }\\
			\frac {ru(k)}{(1-\alpha(1-r))}&\text{if $d(i,j,k)=D+1$. }
		\end{cases}
		\]
		Then, by the second definition of the total variation distance,
		\[
		|| u-q||_{TV}= \sum_{k\in (j):d(i,j,k)=D+1}[u(k)-q(k)] =(1-r)\frac{\alpha(1-\alpha)
		}{1-\alpha(1-r)} \leq (1-r)\leq
		\gamma.\]
		
		By the first definition of the total variation distance,
		
		$$|v-g_D|=|\langle u-q,a\rangle|\leq ||u-q||_{TV}\leq \gamma.$$  
	\end{proof}
	Note that in the above proof, if $r\geq \frac{1}{2}$, it holds that $|| u-q||_{TV}\leq \tfrac{\gamma}{2}$, since 
	\[
	(1-r)\frac{\alpha(1-\alpha)
	}{1-\alpha(1-r)} \leq (1-r)\frac{\alpha(1-\frac{\alpha}{2})^2
	}{1-\frac{\alpha}{2}}\leq \frac{(1-r)}{2}\leq \frac{\gamma}{2}.
	\]
	
	Equation~\eqref{equation L jumps} and Lemma~\ref{lemma J} imply the following corollary.
	\begin{corollary}\label{corollary Lyapunov}
		For every weighted graph $(G,W)$,  agent $i\in V$,  $\gamma\in(0,1)$,  $r\in[1-\gamma,1)$,  $\eta>0$, and $(\gamma,\eta)$-robust $A\in\mathcal H$, the Lyapunov function $L_{i,r}(t)$  is non-increasing in $t$.
	\end{corollary}
	
	We are now ready to prove Proposition~\ref{proposition convergence}. We actually prove a slightly stronger implication than the existence of the designated limit of $(A_{i,2t+\iota})_{\iota=0,1}$. We prove that the variations of $A_{i,2t+\iota}$ are finite. The following definition captures the sum of the two variations. 
	\begin{definition}
		\label{definition variation}
		For $b\geq a \geq 0$, $A\in\mathcal H$, and $i\in V$, we define
		\[
		V_{a}^b[A_{i,t}]:=\sum_{a<t<b}|A_{i,t+1}-A_{i,t-1}|.
		\]
	\end{definition}
	Note that if $V_{a}^\infty[A_{i,t}]<\infty$ for some $a\geq 0$, then Proposition~\ref{proposition convergence} holds.
	
	\begin{lemma}\label{lemma variation}
		For any $\gamma\in(0,1)$, $\eta >0$, $(\gamma,\eta)$-robust $A\in\mathcal H$,
		\[
		V_{a}^b[A_{i,t}]\leq \eta^{-1}\left(L_{i,1-\gamma}(a)-L_{i,1-\gamma}(b-1)\right),
		\]
		for any $b> a \geq 0$.
	\end{lemma}
	Note that Lemmata~\ref{lemma M_i} and \ref{lemma variation} imply Proposition~\ref{proposition convergence}, since $L_{i,1-\gamma}(t)\leq M_i(1-\gamma)<\infty$ for any $t>0$ and, therefore, $V_{0}^\infty[A_{i,t}]\leq\eta^{-1}M_i(1-\gamma)<\infty$.
	
	\begin{proof}[Proof of Lemma~\ref{lemma variation}]
		Fix $a<t<b$. Let  $g_D=\frac{1}{W(i)}\sum_{j\in (i)}W(i,j)A_{j,t}$. Let \[f(x)=\sum_{j\in (i)}W(i,j)(x-A_{j,t})^2= W(i)(x-g_D)^2+C,\] for some $C\in\mathbb R$. Since $J^+_i(t)=f(A_{i,t+1})$ and $J^-_i(t)=f(A_{i,t-1})$, we have 
		\[
		\left(A_{i,t-1}-g_D\right)^2-\left(A_{i,t+1}-g_D\right)^2=\frac 1 {W(i)}\left(J^-_i(t)-J^+_i(t)\right).
		\]
		By Condition (A3) applied with $v=g_D$,
		\[
		|A_{i,t+1}-A_{i,t-1}|\leq \eta^{-1}\left(J^-_i(t)-J^+_i(t)\right).
		\]
		By Lemma~\ref{lemma J} and Equation~\eqref{equation L jumps}, 
		\[
		J^-_i(t)-J^+_i(t)\leq \sum_{j\in V}\left(J^-_j(t)-J^+_j(t)\right)= L_{i,1-\gamma}(t-1)-L_{i,1-\gamma}(t).
		\]
		Since $L_{i,1-\gamma}(t)$ is a non-increasing, by Corollary~\ref{corollary Lyapunov}, summing over all $a<t<b$ we get
		\[
		V_a^b [A_{i,t}]\leq \eta^{-1}\sum_{a<t<b}\left(L_{i,1-\gamma}(t-1)-L_{i,1-\gamma}(t)\right)
		=\eta^{-1}\left(L_{i,1-\gamma}(a)-L_{i,1-\gamma}(b-1)\right).
		\]
	\end{proof}
	\section{Proof of Theorem~\ref{theorem mu}}\label{section proof thm mu}
	In this section we prove Theorem~\ref{theorem mu}. Namely, we provide conditions under which the limiting opinion of an agent  is close to the expected initial opinion with probability close to one.
	The event considered in Theorem~\ref{theorem mu} is the union of two pairs of symmetric events: 
	\[
	\{\Vert Z_i-(\mu,\mu)\Vert>\delta\}=\bigcup_{\iota=0,1}\left(\{\lim_{t\to\infty}A_{i,2t+\iota}<\mu-\delta\}\cup\{\lim_{t\to\infty}A_{i,2t+\iota}>\mu+\delta\}\right).
	\]
	In our proof we consider one of the four events. Similar proofs apply to the other three. 
	Theorem~\ref{theorem mu} follows immediately from the following lemma.
	\begin{lemma}\label{lemma mu}
		Under the conditions, notations, and quantification of Theorem~\ref{theorem mu},
		\[\Pr(\lim_{t\to\infty}A_{i,2t}< \mu-\delta)<\frac \rho 4.\]
	\end{lemma}
	
	The proof of Lemma~\ref{lemma mu} is divided into four sections. In Section~\ref{section degroot}, we provide a few preliminary facts regarding standard DeGroot dynamics. In Sections~\ref{section up to time n} and \ref{section from time n}, we prove Lemma~\ref{lemma mu} under the assumption of undistorted monitoring ($\beta=0$). In Section~\ref{section distorted monitoring}, we extend the proof to the general case.
	
	\subsection{Analysis of DeGroot dynamics}
	\label{section degroot}
	Let $D_{i,t}$ be (standard) DeGroot dynamics on the weighted graph $(W,G=(E,V))$ with i.i.d.\ initial opinions with expectation $\mu$. We analyze the speed in which $D_{i,t}$ converges to $\mu$. To this end, it is a standard technique to consider the voters model $V_{i,t}$ defined as follows:
	\[
	V_{i,0}:=D_{i,0},\quad\forall i\in V.
	\]
	In each time period $t>0$, $V_{i,t}$ is chosen randomly by
	\[
	\Pr(V_{i,t}=V_{j,t-1})=\frac {W(i,j)} {W(i)},\quad\forall j\in V,
	\]
	where $\{V_{i,t}\}_{i\in V}$ are independent conditional on $\mathcal H_{<t}$.
	
	By induction on $t$, one has that 
	\begin{equation}
		\label{eq: voter}
		D_{i,t}=E[V_{i,t}|\mathcal H_0].
	\end{equation}
	Let $(R_t)_{t=0}^\infty$ be a standard random walk on $(W,G)$ originating in $i$. Conditional on $\mathcal H_0$, the distribution of $V_{i,t}$ is identical to the distribution of $V_{R_{t},0}$ (again, by induction on $t$). Let $p_t=\max\{\Pr(R_t=j):j\in V\}$. The next lemma shows that
	\begin{equation}\label{equation pt}
		p_t=\mathcal O\left(\frac {\overline W} {\sqrt t}\right).
	\end{equation}
	
	\begin{lemma}\label{lemma random walk}
		There exists a global constant $C$, such that for any $t>0$ and $j\in V$ whose eccentricity is at least $t$ (in particular, if $G$ is infinite and connected then the eccentricity of any vertex is infinite),
		\[
		\Pr(R_t=j)\leq\frac{W(j) \cdot C}{\sqrt{t}}.
		\]
	\end{lemma}
	\begin{proof}
		Our proof is based on the method of evolving sets \cite{morris2005evolving}. A different approach in a slightly different setting can be found in  \cite{barlow2019random}.
		
		We prove that the probability that $R_t=j$ can be bounded from above by the probability that a standard Brownian motion beginning at $W(j)$ does not hit 0 by time $\Omega\left(t\right)$.
		
		We describe a random walk by a sequence of i.i.d.\ random functions $g_1, g_2,...$, where  each $g_t:V\rightarrow V$ maps vertex $i$ to $j$ with probability $\frac {W(i,j)}{W(i)}$ independently of how it maps the other vertices.
		
		The random walk $\{R_t\}_{t=0}^\infty$ originating at $R_0=i$ is realized as
		
		\[R_t=g_t\circ g_{t-1} \circ ... \circ g_1(i).\]
		
		Thus, the event $\{R_t=j\}$ can be expressed as
		\[\left\{R_t=j\right\}=\left\{i\in \left(g_t\circ ...\circ g_1\right)^{-1}(j)\right\}=\left\{i\in g_1^{-1}\circ g_2^{-1}\circ...\circ g_t^{-1}(j)\right\}.\]
		
		Therefore,
		\[\Pr(R_t=j)\leq \Pr(g_1^{-1}\circ g_2^{-1}\circ...\circ g_t^{-1}(j)\neq\emptyset).\]
		
		For every vertex $v\in V$ and every finite set of vertices $A\subset V$, define $W(v,A):=\sum_{a\in A} W(v,a)$ and $W( A):=\sum_{a\in A} W(a)= \sum_{v\in V}W(v, A)$. 
		Since $g_1,\ldots,g_t$ are i.i.d.\ we can change their order without changing their distribution. It will be convenient to consider the inverse order, namely,
		\[
		\Pr(R_t=j)\leq \Pr(W(g_1^{-1}\circ...\circ g_t^{-1}(j))>0)=\Pr(W(g_t^{-1}\circ ...\circ g_1^{-1}(j))>0).\]
		
		Define a sequence of random sets of vertices  $A_0,A_1,\ldots\subset V$ recursively by\footnote{The random sets $A_0,A_1,\ldots$ are known as ``evolving sets'' due to \cite{morris2005evolving}.}
		\[A_0=\left\{j \right\},~
		A_{k+1}=g_{k+1}^{-1}(A_k).
		\]
		
		Define $X_k=W(A_k)$. We have
		\begin{multline*}
			E\left[X_k | A_{k-1}\right]=\sum_{v\in V} W(v)\Pr(g_k(v)\in A_{k-1})=  \sum_{v\in V} W(v) \frac{W(v,A_{k-1})}{W(v)}=\\
			\sum_{v\in V}W(v,A_{k-1}) = W(A_{k-1}).
		\end{multline*}
		Namely, $X_k$ is a non-negative martingale, and 
		\[\Pr\left(R_t=j\right)\leq \Pr(X_t > 0) = \Pr\left(X_1,...,X_t > 0 \right).\]
		
		Let $\{B_s\}_{s\geq 0}$ be a standard Brownian motion  originating at $B_0=W(j)$. We construct stopping times $s_1<s_2<\cdots$ such that for every $n\in\mathbb N$, 
		\[(X_1,...,X_n)\sim (B_{s_1}, B_{s_2}, ..., B_{s_n}),\] conditional on the event $\{X_n>0\}=\{X_1,\ldots,X_n>0\}=\{B_s>0:0\leq s\leq s_n\}$.
		
		Let $s_0=0$. For $k>0$, conditioning on $A_{k-1}$ we construct $\hat s_k= s_k-s_{k-1}$. Let us represent the martingale increment $X_k-X_{k-1}$ as a sum of independent random variables. Let \[Y_v^k:=W(v)\left( \mathbf 1_{\left(g_k(v)\in A_{k-1}\right)} -  \Pr \left(g_k(v)\in A_{k-1}\right)\right).\] Then,
		\[X_k-X_{k-1}=\sum_{v\in V}Y_v^k.\]
		In the above sum, for all but a finite set of vertices, we have $Y_v^k\equiv 0$, and for the rest of the vertices,
		$\left\{Y_v^k\right\}_{v}$ are independent,  take 2 values,  $a_v<0<b_v$, and $E[Y_v^k]=0$. 
		
		As long as $A_{k-1}$ is nonempty, there is at least one vertex $v$ such that $|Y_v^k|\geq 1$, namely,  $a_v\leq -1$ and $b_v\geq 1$. To see this, note that by induction on $k$, $A_{k-1}\subset B(j,k-1)$, and since the eccentricity of $j$ is assumed to be at least $t>k-1$, there exists a vertex $v$ that has at least one neighbor in $A_{k-1}$ and one neighbor outside $A_{k-1}$. For such a vertex $v$, it holds that $\frac 1 {W(v)}\leq \Pr(g_k(v)\in A_{k-1})\leq 1- \frac 1 {W(v)}$ and, therefore, $|Y^k_v|\geq 1$.
		
		Each non-trivial $Y_v^k$ (i.e., $Y^k_v\not\equiv 0$) can be implemented by a stopping time of a standard Brownian motion that begins at 0 and stops when it gets to one of the two values, $a_v$ or $b_v$. The sum $\sum_{v\in V}Y^k_v$ can be implemented by concatenating such stopping times to obtain a stopping time, which is how we define $\hat{s}_k$.

		Since there exists $v$ such that $|Y_v^k|\geq 1$, there are independent random variables $\hat \tau_1,\hat \tau_2,\ldots$ that are distributed according to a globally fixed distribution with a support in $(0,1]$ such that $\tau_k\leq \hat s_k$. For example, $\hat\tau_k$ can be the minimum between 1 and the time until a standard Brownian motion originating at 0 hits $\{-1,+1\}$.
		
		If $A_{k-1}=\emptyset$, define $\hat{s}_k:=\hat\tau_k$. Thus, 
		\[
		s_k=\hat{s}_1+\hat{s}_2+...+\hat{s}_k\geq \hat{\tau}_1+\hat{\tau}_2+...+\hat{\tau}_k.
		\]
		By Hoeffding's inequality \cite{hoeffding1963probability}, \[\Pr(s_k<k E[\hat \tau_1 ]/2)\leq \exp(-kE[\hat \tau_1]^2/8)=\exp(-\Omega(k)).\] 
		Given that $B_{s_{k-1}}>0$, we have $W(A_{k-1})=X_{k-1}=B_{s_{k-1}}$ and, therefore, $A_{k-1}\neq \emptyset$. It follows that $B_{s_k}=X_{k}\geq 0$. If in addition $X_k=0$, then $A_k=\emptyset$ and it must be  that all of the non-trivial $\{Y_v^k\}_v$ are negative. In this case $B_s>0$ for all $s\in (s_{k-1},s_{k})$.  It follows that
		
		\begin{multline*}
			\Pr\left(B_{s_1},B_{s_2},...,B_{s_k}>0\right)=\Pr\left(B_s>0: s\in [0,s_k]\right)\\
			\leq \Pr\left(B_s>0: s\in [0,kE[\hat\tau_1]/2]\right)+\Pr(s_k<kE[\hat\tau_1]/2)\\=\Pr\left(B_s>0: s\in [0,kE[\hat\tau_1]/2]\right)+\exp(-\Omega(k)).
		\end{multline*}
		By the reflection principle,
		\[
		\Pr\left(B_s>0: s\in [0,a]\right)=\Pr(|B_{a}-B_0|<B_0)\leq 2B_0\frac 1 {\sqrt{2\pi a}}.
		\]
		The proof of the lemma is concluded by letting $a=kE[\hat\tau_1]/2$, $k=t$ and $B_0=W(j)$.
	\end{proof}
	
	Since $D_{i,t}=\sum_{j\in V}\Pr(R_t=j)D_{j,0}$, by Hoeffding's inequality \cite{hoeffding1963probability}, we have
	\[
	\Pr(D_{i,t}<\mu-\delta)\leq \exp\left(\frac {-\delta^2}{2\sum_{j\in V}P^2(R_t=j)}\right)\leq \exp\left(\frac{-\delta^2}{2p_t}\right).
	\]
	By \eqref{equation pt}, we obtain the following corollary.
	\begin{corollary}\label{corollary hoeffding}
		Let $(G,W)$ be a connected weighted graph with $\overline W<\infty$ radius at least $t$. Let $D_{i,t}$ be the DeGroot dynamics on $(G,W)$ with $\{D_{i,0}\in[0,1]\}_{i\in V}$ i.i.d.\ and $E[D_{i,0}]=\mu$. Then, \[
		\Pr\left(D_{i,t}< \mu-\frac \delta 3\right)\leq \exp\left(\frac{-C_0\delta^2\sqrt{t}}{\overline W}\right)
		\]
		for all $i\in V$ and some universal constant $C_0>0$. \end{corollary}
	
	Let $n=\left\lfloor\frac{\delta}{3\varepsilon}-1\right\rfloor$. It follows that
	\[
	\Pr\left(D_{i,n}< \mu- \frac \delta 3\right)\leq \exp\left(-C_1\left(\frac{\delta^{2.5}}{d\sqrt{\varepsilon}}\right)\right),
	\]
	for some universal $C_1>0$. Define \[\rho_1=\rho_1(d,\varepsilon,\delta):=\exp\left(-C_1\left(\frac{\delta^{2.5}}{d\sqrt{\varepsilon}}\right)\right).\]

	\subsection{Up to time n}
	\label{section up to time n}
	Let $(G,W)$ be a weighted graph that satisfies the conditions of Theorem~\ref{theorem mu}. Let $g=\{g_{i,t}\}_{i\in V,t\geq 0}$ be an $(\varepsilon,\gamma,\eta)$-robust updating rule that induces the history $A\in\mathcal H$ with undistorted monitoring.
	
	Let $R_1=n$. We will prove Lemma~\ref{lemma mu} with some $R>R_1$. Therefore, until time $n$, agent $i$ is not influenced by any stubborn agent. We couple $A_{j,t}$ with DeGroot dynamics $D_{j,t}$ by setting $A_{j,0}=D_{j,0}$, for all $j\in V$. By Assumptions (A1) and (A2), and by induction on $t$, for any $t\leq n+1$,
	\[
	A_{i,t}\geq D_{i,t}-t\varepsilon\geq D_{i,t}-(n+1)\varepsilon\geq D_{i,t}- \frac{\delta}3.
	\]
	It follows that
	\begin{equation}\label{equation A_n}
		\Pr\left(A_{i,n}< \mu-\tfrac{2}3 \delta\right),\Pr\left(A_{i,n+1}< \mu-\tfrac{2}3 \delta\right)\leq \rho_1
	\end{equation}
	for any $i$ that is at least $R_1$ away from any stubborn agent.
	\subsection{From time n onward}
	\label{section from time n}
	For every $j\in V$ and $t\leq n+1$, define 
	\[
	A'_{j,t} = \begin{cases}
		\mu - \frac 2 3 \delta& A_{j,t}\geq \mu - \frac 2 3 \delta,\\
		0& A_{j,t} < \mu - \frac 2 3 \delta.
	\end{cases}
	\]
	For $t>n+1$, define $A'_{j,t}=g_{j,t}(A'_{(j),<t})$.
	
	By Theorem~\ref{theorem convergence}, the limits $z_i:=\lim_{t\to\infty}A_{i,2t}$ and $z'_i:=\lim_{t\to\infty}A'_{i,2t}$ exist. By Condition (A1), since $A'_{\leq n}\leq A_{\leq n}$, we have $A'_t\leq A_t$ for $t>n$ as well. Therefore, $z'_i\leq z_i$. It follows that
	\begin{multline*}
		\Pr(z_i<\mu-\delta)\leq \Pr(z'_i<\mu-\delta)\leq\\
		\leq \Pr\left(A'_{i,n}<\mu-\tfrac 2 3 \delta\right)+\Pr\left(|z'_{i}-A'_{i,n}|>\tfrac 1 3 \delta\right)\leq\\
		\leq \rho_1 + \Pr\left(V_n^\infty [A'_{i,t}]>\tfrac 1 3 \delta\right).
	\end{multline*}
	Let $L'(t)=L'_{i,1-\gamma}(t)$ be the Lyapunov function defined w.r.t.\ the history $A'$. 
	By Lemma~\ref{lemma variation} and Markov's inequality,
	\[
	P\left(V_n^\infty [A'_{i,t}]>\tfrac 1 3 \delta\right)\leq P\left(L'(n)> \tfrac {\eta\delta} 3\right)\leq \tfrac 3 {\eta\delta}\mathbb E[L'(n)].
	\]
	Summing up, we get 
	\begin{equation}\label{eq pzi}
		\Pr(Z_i<\mu-\delta)\leq \rho_1 + \tfrac 3 {\eta\delta}\mathbb E[L'(n)].
	\end{equation}
	
	It remains to bound $\mathbb E[L'(n)]$ from above. Let $\Delta_i \geq R\geq R_1$ be the distance between $i$ and the nearest stubborn agent (if there are no stubborn agents, then $\Delta_i=\infty$). Recall that $w(j,k):=W(j,k)\cdot (1-\gamma)^{d(i,\{j,k\})}$ and $M_i(1-\gamma):= \sum_{j\in V}\sum_{k\in(j)}w(j,k)$. Let $\widehat M_i(1-\gamma):= \sum_{j\in V}(1-\gamma)^{[d(i,j)-1]_+}\geq \frac {M_i(1-\gamma)}{\overline W}$. The proof of Lemma~\ref{lemma M_i} provides a stretched sub-exponential function $g(\tfrac{1}{\gamma})\geq \widehat M_i(1-\gamma)$.
	By \eqref{equation A_n},
	\begin{multline}\label{equation E[L']}
		\mathbb E[L'(n)]=\sum_{j,k\in V}w(j,k)\mathbb E\left[(A'_{j,n}-A'_{k,n+1})^2\right]\\
		\leq\sum_{\substack{j\in B(i,R-R_1).
				\\k\in(j)}}w(j,k)\left(\Pr(A'_{j,n}<\mu-\tfrac 2 3\delta )+\Pr(A'_{k,n+1}<\mu-\tfrac 2 3\delta )\right)+\sum_{\substack{j\not\in B(i,R-R_1)\\k\in(j)}}w(j,k)\leq\\
		\leq 2\rho_1 M_i(1-\gamma) + \overline{W}\sum_{j:d(i,j)\geq R-R_1+1}(1-\gamma)^{d(i,j)-1}\\
		\leq \overline W\left(2\rho_1  g\left(\frac 1 \gamma\right) + \sum_{j:d(i,j)\geq R-R_1+1}(1-\gamma)^{d(i,j)-1}\right).
	\end{multline}
	
	We now bound the term $\sum_{j:d(i,j)\geq R-R_1+1}(1-\gamma)^{d(i,j)-1}$ from above. Since $(1-\gamma)\leq \left(1-\tfrac{\gamma}{X}\right)^X$, for every $X\geq 1$, 
	\begin{multline}\label{equation tail M}
		\sum_{j:d(i,j)\geq R-R_1+1}(1-\gamma)^{d(i,j)-1} \leq \sum_{j:d(i,j)\geq R-R_1+1}\left(1-\frac{\gamma}{R-R_1}\right)^{(R-R_1)(d(i,j)-1)}
		\\
		\leq \left(1-\frac{\gamma}{R-R_1}\right)^{(R-R_1)(R-R_1-1)}\widehat M_i\left(1-\frac{\gamma}{R-R_1}\right)\\
		\leq \exp(-\gamma(R-R_1-1))g\left(\frac{R-R_1}{\gamma}\right).
	\end{multline}
	
	Plugging \eqref{equation tail M} into \eqref{equation E[L']} and the resulting expression into \eqref{eq pzi}, we get that Lemma~\ref{lemma mu} holds for any pair $\rho,R>0$ that satisfies the relation
	\[
	\frac \rho 4 = \rho_1+\tfrac{3\overline W}{\eta \delta}\left(2\rho_1  g\left(\frac 1 \gamma\right) + \exp(-\gamma(R-R_1-1))g\left(\frac{R-R_1}{\gamma}\right)\right).
	\]
	If we fix $\delta$, $\overline{W}$, and $f$ (and hence $g$), and suppose that $\eta^{-1}+\gamma^{-1}\leq poly(\varepsilon^{-1})$, then we can take
	\[
	R= \gamma^{-1.00001},
	\]
	and we have $\rho\to 0$ as $\varepsilon\to 0$. Furthermore, $\rho^{-1}$ is a stretched exponential function of $\varepsilon^{-1}$. To see this, recall that $R_1\leq \tfrac{\delta}{3\varepsilon}$, $\gamma\leq \varepsilon$ and $\rho_1=\exp\left(-\Omega\left(\frac{\delta^{2.5}}{\overline{W}\sqrt{\varepsilon}}\right)\right)$. Since $\rho_1^{-1}$ and $\exp(\gamma(R-R_1-1))$ are stretched exponential functions of $\varepsilon^{-1}$ and $g\left(\frac 1 \gamma\right)$, $g\left(\frac{R-R_1}{\gamma}\right)$ and $\tfrac{3\overline W}{\eta \delta}$ are stretched sub-exponential functions of $\varepsilon^{-1}$, we have that $\rho^{-1}$ is a stretched exponential function of $\varepsilon^{-1}$.

	This completes the proof of Lemma~\ref{lemma mu} for the case $\beta=0$.
	\subsection{Distorted monitoring}\label{section distorted monitoring}
	We turn now to prove Lemma~\ref{lemma mu} for any $\beta\in [0,\gamma)$. For any $A\in \mathcal H$,  define $A^{\beta}$, by
	\[
	A_{i,t}^{\beta} = \max\{0,A_{i,t}-\beta\}, \quad\forall i\in V\ \forall t\geq 0.
	\]
	For an updating rule $g$, define $g^\beta$ by \[
	g_{i,t}^\beta(A_{(i),<t})=g_{i,t}\left(A^\beta_{(i),<t}\right).
	\]
	By Lemma~\ref{lemma stubborn and bias reduction}, if $g$ is $(\varepsilon,\eta,\gamma)$-robust, then $g^\beta$ is $(\varepsilon+\beta,\eta,\gamma-\beta)$-robust. Let $A_{i,t}$ be a dynamics induced by $g$ with $\beta$-distorted monitoring, and let $A'_{i,t}$ be the dynamics induced by $g^\beta$ with the same initial opinions and with undistorted monitoring. Let $z'_i=\lim_{t\to\infty} A'_{i,2t}$ and $z_i=\lim_{t\to\infty} A_{i,2t}$. By induction on $t$, Condition (A1) implies that $A'_{i,t}\leq A_{i,t}$ for all $i\in V$ and $t>0$. Therefore, $z'_i\leq z_i$ for all $i\in V$. By Lemma~\ref{lemma mu}, for the case $\beta=0$, 
	\[
	\Pr(z_i<\mu-\delta )\leq \Pr(z'_i<\mu-\delta)\leq \tfrac 1 4\rho(\overline{W},\varepsilon+\beta,\eta,\gamma-\beta,\delta,f,0)
	\]
	for any $i$ that is at least $R(\overline{W},\varepsilon+\beta,\eta,\gamma-\beta,\delta,f,0)$ away from any stubborn agent. Therefore, letting 
	\begin{align*}
		R(\overline{W},\varepsilon,\eta,\gamma,\delta,f,\beta) &= R(\overline{W},\varepsilon+\beta,\eta,\gamma-\beta,\delta,f,0),\\
		\rho(\overline{W},\varepsilon,\eta,\gamma,\delta,f,\beta) &= \rho(\overline{W},\varepsilon+\beta,\eta,\gamma-\beta,\delta,f,0)
	\end{align*}
	provides the requirements of Lemma~\ref{lemma mu}.
	
	This completes the proof of Lemma~\ref{lemma mu} and, therefore, of Theorem~\ref{theorem mu} as well.
	
	\section{Additional proofs}
	\begin{proof}[Proof of Theorem~\ref{theorem:golubjackson}]
		Let $G=(V,E)$ be an infinite connected graph of bounded degree. And let $D_{i,t}$ be a standard DeGroot dynamics on $G$ with bounded i.i.d. initial opinions with expectation $\mu$. By scaling we may assume w.l.o.g.\ that $D_{i,t}\in[0,1]$. By applying Corollary~\ref{corollary hoeffding} to both $D_{i,t}$ and $1-D_{i,t}$, we get
		\[
		\sum_{t=1}^\infty\Pr(|D_{i,t}-\mu|>\delta)<\infty,
		\]
		for every $i\in V$ and $\delta>0$.
		By the Borel--Cantelli lemma, we get that $D_{i,t}$ converges to $\mu$ almost surely, for every $i\in V$.
	\end{proof}
	\begin{proof}[Proof of Lemma~\ref{lemma bots}]
		Let $G=(V,E)$ be a finite connected graph. Let the variables $(D_{i,0})_{i\in V,i\neq i_0}$ be any initial opinions for all the agents but agent $i_0$. We assume that all agents, but agent $i_0$, apply the DeGroot dynamics. That is, we assume $D_{i_0,t}=c$ is constant at any time and that $D_{i,t}$ is the standard DeGroot dynamics on $G$. 
		
		Let $R_{i,t}$ be a standard random walk started at node $i$ such that node $i_0$ is an absorbing state. That is, when the random walk enters $i_0$ it stays there and never exits.
		By induction on $t$, for any $i\in V$ one has that $D_{i,t}=\sum_{j\in V}P(R_{i,t}=j)D_{j,0}$. 
		Since $G$ is a connected and finite graph it follows that 
		$\lim_{t\to\infty}P(R_{i,t}=i_0)=1$ for every $i\in V$.
		Therefore $\lim_{t\to\infty}D_{i,t}=D_{i_0,0}$ for every node $i$.
	\end{proof}
	
	{
		\section{Example: divergence of standard DeGroot}\label{section example divergence}
		
		We present an infinite graph $G=(V,E)$, a configuration of initial opinions $\{D_{i,0}\}_{i\in V}$ and an agent $i_0\in V$ such that the standard DeGroot dynamics on $G$ diverges at $i_0$, namely, the limit $\lim_{t\to \infty}D_{i_0,2t}$ does not exist.   
		Consider the line graph $\mathbb Z$. Namely, $V= \mathbb Z$ and $E=\{(n,n+1):n\in \mathbb Z\}$. Let $i_0 = 0$. By \eqref{eq: voter} and the sentence that follows it, we have
		\[
		D_{0,2t}=\sum_{j=-t}^t\binom{2t}{j+t}2^{-2t}D_{2j,0}.
		\]
		Let $t_1<t_2<\cdots$ be a sequence of natural numbers such that 
		\[
		\lim_{n\to\infty}\sum_{t_n<|j|\leq t_{n+1}}\binom{2t_{n+1}}{j+t_{n+1}}2^{-2t_{n+1}}=1.
		\]
		For $|j|\leq t_1$, set $D_{j,0}:=0$, and for $t_n<|j|\leq t_{n+1}$ set 
		\[
		D_{j,0}:=\begin{cases}
			1&\text{$n$ odd,}\\
			0&\text{$n$ even.}
		\end{cases}
		\]
		It follows that $\lim_{n\to\infty}D_{0,2t_{2n}}=1$, whereas, $\lim_{n\to\infty}D_{0,2t_{2n+1}}=0$. Hence, $\lim_{t\to\infty}D_{0,2t}$ does not exist.
	}
\end{document}